\documentclass[12pt,reqno,oneside]{amsart}

\usepackage[totalwidth=16.5cm,totalheight=22cm]{geometry}
\usepackage{amsmath}

\usepackage{amsthm}
\usepackage{graphicx}
\usepackage{lscape}
\usepackage[all]{xy}

\usepackage{adjustbox}
\usepackage{multicol}
 
\usepackage{colortbl}
\usepackage{color}
\usepackage{epsfig}
\usepackage[normalem]{ulem}

\usepackage{amssymb}
\usepackage{latexsym}
\newtheorem{Definition}{Definition}[section]
\newtheorem{Theorem}[Definition]{Theorem}
\newtheorem*{Theorem*}{Theorem}
\newtheorem{Proposition}[Definition]{Proposition}
\newtheorem{Lemma}[Definition]{Lemma}
\newtheorem{Corollary}[Definition]{Corollary}

\newtheorem{case}{Case}

\theoremstyle{remark}
\newtheorem{Remark}[Definition]{Remark}
\newtheorem{example}[Definition]{Example}

\def\NN{{\mathcal N}}

\def\O{{\mathrm{O}}}
\def\SO{\mathrm{SO}}

\def\mod{{\mathrm{mod}}}

\def\Norm{\mbox{Norm}}

\def\C{{\mathbb C}}
\def\R{{\mathbb R}}
\def\H{{\mathbb H}}
\def\Z{{\mathbb Z}}
\def\Q{{\mathbb Q}}

\newcommand{\red}{\textcolor{red}}

\begin{document}

\title{On the diffeomorphism type of Seifert fibered spherical 3-orbifolds}

\author[M. Mecchia]{Mattia Mecchia}
\address{M. Mecchia: Dipartimento Di Matematica e Geoscienze, Universit\`{a} degli Studi di Trieste, Via Valerio 12/1, 34127, Trieste, Italy.} \email{mmecchia@units.it}
\author[A. Seppi]{Andrea Seppi}
\address{A. Seppi: CNRS and Universit\'e Grenoble Alpes, 100 Rue des Math\'ematiques, 38610 Gi\`eres, France.} \email{andrea.seppi@univ-grenoble-alpes.fr}

\thanks{The first author was partially supported by the FRA 2018 grant  ``Aspetti geometrici, topologici e computazionali delle variet\`{a}'', Universit\`{a} degli Studi di Trieste. The authors are members of the national research group GNSAGA}



\maketitle

\begin{center}
\textit{Dedicated to Bruno Zimmermann on his 70th birthday}
\end{center}

\begin{abstract}
It is well known that, among closed spherical Seifert three-manifolds, only lens spaces and prism manifolds admit several Seifert fibrations which are not equivalent up to diffeomorphism. Moreover the former admit infinitely many fibrations, and the latter exactly two. In this work, we analyse the non-uniqueness phenomenon for orbifold Seifert fibrations. For any closed spherical Seifert three-orbifold, we determine the number of its inequivalent fibrations. When these are in a finite number (in fact, at most three) we provide a complete list. In case of infinitely many fibrations, we describe instead an algorithmic procedure to determine whether two closed spherical Seifert orbifolds are diffeomorphic.
 \end{abstract}

\begin{center}

\end{center}

\section{Introduction} 

\emph{Seifert fibered  $3$-manifolds} were introduced by Seifert \cite{seifert} and are one of the cornerstones  in the study of 3-dimensional  manifolds  (see for example \cite{orlik, scott, boileau-maillot-porti}).   Roughly speaking a Seifert fibered $3$-manifold is a fiber bundle whose fibers are circles, except that some fibers are \emph{exceptional}, meaning that a tubular neighborhood is a torus which is however not fibered as a  product. The presence of singular fibers is reflected by the fact that the base of the fibration is not a manifold, but a particular type of $2$-dimensional orbifold, namely a surface containing \emph{cone points}.

Orbifolds are a generalization of manifolds, which had been introduced in different contexts by Satake \cite{satake}, by Thurston \cite[Chapter 13]{thurston2} and by Haefliger \cite{haefliger} -- useful references being also \cite{boileau-maillot-porti, choi, Dun2, scott}. The most standard example of an orbifold (of dimension $n$) is the quotient of a manifold $M^n$ by a group $\Gamma$ which acts properly discontinuously -- but in general not freely -- on $M$. If the action is not free, \emph{singular points} appear in the quotient $M/\Gamma$, keeping track of the action of point stabilizers $\mathrm{Stab}_\Gamma(x)$ on a neighborhood of a fixed point $x\in M$. More generally, an orbifold is \emph{locally} the quotient of a manifold by the action of a finite group. 

Bonahon and Siebenmann  \cite{bonahon-siebenmann} generalized the definition of Seifert fibration to $3$-orbifolds. This notion is actually more general than its counterpart for $3$-manifolds: here fibers are in general quotients of the circle, and therefore some exceptional fibers are allowed to be intervals, which corresponds to the fact that the base of the fibration has singularities which are not of conical type.  Seifert fibered $3$-manifolds and $3$-orbifolds are classified up to fibration-preserving diffeomorphisms by Seifert invariants, but some of them admit more than one Seifert fibration. The complete topological classification of the Seifert fibered $3$-manifolds follows from the work of several authors (see \cite[Section 2.4.1]{boileau-maillot-porti}). In the present paper we give a classification by diffeomorphism type of closed Seifert fibered \emph{spherical}  $3$-\emph{orbifolds}.  {Throughout the paper  $3$-orbifolds and $3$-manifolds are assumed to be orientable; indeed, in the spherical fibered case this assumption is not a restriction since non-orientable spherical $3$-orbifolds never admit a Seifert fibration, see Corollary \ref{cor nonoriented}.}

Spherical $3$-orbifolds are one of the eight classes of  \emph{geometric} $3$-\emph{orbifolds}, which had a large importance in Thurston's geometrization program.  These are \emph{locally} the quotient of one of the eight Thurston's model geometries by the properly discontinuous action of a group of isometries.  In the case of manifolds,  six of eight Thurston's geometries give Seifert fibered $3$-manifolds (the exceptions are {hyperbolic and Sol geometries}). In the orbifold setting the situation is the same with few exceptions: twelve euclidean $3$-orbifolds and eighteen spherical $3$-orbifolds are not fibered, see \cite[Theorem1]{Dun2}.  The eighteen  spherical  $3$-orbifolds not admitting a Seifert fibration are analyzed by Dunbar \cite{dunbar}. On the other hand each  closed Seifert fibered $3$-orbifold without bad $2$-suborbifold admits a geometric structure \cite[Proposition 2.13]{boileau-maillot-porti}; \emph{bad} means that the orbifold  is not globally the quotient of a manifold.

The $3$-manifolds admitting multiple fibrations are either spherical, euclidean  or covered by the Thurston geometry $S^2 \times \mathbb{R},$ see \cite{scott}. In the euclidean case there exists a unique $3$-manifold admitting {several} inequivalent fibrations and in the $S^2 \times \mathbb{R}$ case we find two of these manifolds. The most intriguing case from this point of view is the spherical one where manifolds belonging to two important classes (lens spaces and prism  manifolds) admit multiple fibrations. The situation for Seifert fibered $3$-orbifolds is similar {in the sense that interesting non-uniqueness phenomena mostly appear for spherical geometry}; see for example \cite[Theorem 2.15]{boileau-maillot-porti}.  

In this paper we compute, for each possible closed fibered spherical $3$-orbifolds, {the number of} fibrations it admits. {When}  the number of  the fibrations is finite we describe explicitly the Seifert invariants of all fibrations of any $3$-orbifold. When infinitely many fibrations occur, we describe an algorithm to decide if two sets of Seifert invariants give the same $3$-orbifold. Hence we obtain  a complete classification of closed Seifert fibered spherical $3$-orbifold up to orientation-preserving diffeomorphism. In fact, {for orientable Seifert fibered $3$-orbifolds the Seifert invariants are complete invariants of oriented fibered $3$-orbifolds. We also remark that we do not need to assume that the singular locus is non-empty, hence the results of this paper also hold true when $\mathcal O$ is a manifold. As already mentioned, the manifold case is well-known from the literature.} 

The  presentation of all possibile diffeomorphisms is rather technical and  summarizing it  in a single statement seems impossible to us. In the present introduction we describe the situation in terms of number of fibrations which are admitted. The complete description of the classification  can be found in Section~\ref{classification}. 

\begin{Theorem} \label{main thm intro}
Let $\mathcal O$ be a closed spherical Seifert fibered $3$-orbifold with base orbifold $\mathcal B$ and $b$ an integer greater than one.

\begin{enumerate}
\item If $\mathcal B\cong S^2(2,2,b), \,D^2(b), \, \mathbb{R}P^2(b), \, D^2(2;b)$ or   $D^2(;2,2,b)$ then $\mathcal O$ admits \textbf{two} inequivalent fibrations with the following exceptions:
\begin{itemize}
\item $\left(S^2(2,2,b);\,\frac{0}{2},\frac{0}{2},\pm\frac{2}{b};\,\mp\frac{2}{b}\right)\!,$ $\left(S^2(2,2,b);\,\frac{0}{2},\frac{1}{2},\pm\frac{1+b/2}{b};\,\mp\frac{1}{b}\right)\! \text{with}\,b \,\text{even,}  $ \\ $\left(D^2(2;);\pm\frac{b}{2};\,;\mp\frac{b}{2};0 \right)\!,$ $\left(D^2(2;b);\,\frac{1}{2};\,\pm \frac{1}{b};\,\mp\frac{1}{2b};\,1\right)$ and  $ \left(D^2(;2,2,b);;\,\frac{1}{2},\frac{1}{2},\pm\frac{1}{b};\,\mp\frac{1}{2b};1\right)$ which admit \textbf{three} fibrations;

\item $\left(S^2(2,2,b);\,\frac{0}{2},\frac{0}{2},\pm\frac{1}{b};\,\mp\frac{1}{b}\right)\!,$ $\left(S^2(2,2,b);\,\frac{0}{2},\frac{1}{2},\pm\frac{(1+b)/2}{b};\,\mp\frac{1}{2b}\right)\! \text{with}\,b \,\text{odd,}$  \\ $\left(D^2(b;);\pm \frac{1}{b};\,;\mp\frac{1}{b};0\right)\!,$ $ \left(D^2(b;);\pm \frac{(1+b)/2}{b};\,;\mp\frac{1}{2b};1\right)\! \text{with}\,b \,\text{odd,}$ $\left(\mathbb{R}P^2(b);\pm \frac{1}{b};\mp\frac{1}{b}\right)\!,$ \\ $\left(D^2(2;b);\,\frac{0}{2};\,\pm \frac{1}{b};\,\mp\frac{1}{2b};\,1\right)$ with $b$ even,  $\left(D^2(;2,2,b);\,;\,\frac{0}{2},\frac{0}{2},\pm\frac{1}{b};\,\mp\frac{1}{2b};0\right)$ with $b$ odd \\ and  $\left(D^2(;2,2,b);\,;\,\frac{0}{2},\frac{1}{2},\pm\frac{(b+1)/2}{b};\,\mp\frac{1}{4b};1\right)$ with $b$ odd  which admit \textbf{infinitely many} fibrations.
\end{itemize}

\item If $\mathcal B \cong S^2(2,3,b)$ or $D^2(;2,3,b)$ with $b=3,4,5$  then $\mathcal O$ admits a \textbf{unique}  fibration with the following exceptions:

\begin{itemize}
\item $\left(S^2(2,3,3)\,;\frac{0}{2}\,,\pm \frac{2}{3}\,,\pm \frac{2}{3}\,;\mp \frac{1}{3}\right)\!,$   
$\left(S^2(2,3,4)\,;\frac{0}{2}\,,\pm \frac{2}{3}\,,\pm \frac{2}{4}\,;\mp \frac{1}{6}\right)\!,$\\
$\left(S^2(2,3,4)\,;\frac{0}{2}\,,\pm \frac{1}{3}\,,\pm \frac{3}{4}\,;\mp \frac{1}{12}\right)\!,$ 
$\left(S^2(2,3,5)\,;\frac{0}{2}\,,\pm \frac{2}{3}\,,\pm \frac{2}{5}\,;\mp\frac{1}{15}\right)\!,$ \\
 $\left(D^2(;2,3,3)\,;;\frac{1}{2}\,,\pm \frac{1}{3}\,,\pm \frac{1}{3};\mp\frac{1}{12};1\right) \!,$ $\left(D^2(;2,3,4)\,;;\frac{1}{2}\,,\pm \frac{1}{3}\,,\pm \frac{1}{4};\mp \frac{1}{24};1\right)$ \\ and  $\left(D^2(;2,3,5)\,;;\frac{1}{2}\,,\pm \frac{1}{3}\,,\pm \frac{1}{5};\mp\frac{1}{60};1\right)$ which admit \textbf{two} fibrations. 
\end{itemize}

\item    If $\mathcal B $ is   a $2$-sphere  with at most two cone points or a  $2$-disk with at most two corner points, then $\mathcal O$ admits  \textbf{infinitely many}  fibrations.

\end{enumerate}

\end{Theorem}

For the notations for $2$-orbifolds and for fibered $3$-orbifolds we refer the reader to Subsection \ref{subsec:orbifolds} and \ref{sec fibrations}, respectively. We remark that in Theorem \ref{main thm intro}, the list of exceptions sometimes contains the Seifert invariants of different fibrations of the same $3$-orbifold. 
As already said, in Section~\ref{classification} we actually list, for every spherical orbifold admitting multiple inequivalent fibrations as in Theorem \ref{main thm intro}, all of its other fibrations.

We discuss briefly  some topological aspects relating to the theorem. A Seifert fibered $3$-orbifold with base orbifold  a $2$-sphere with at most two cone points has a lens space as underlying topological space and the singular set is a subset of the union of the cores of the two tori giving the lens space; these orbifolds  can be considered as the generalization of lens spaces in the setting of orbifolds. {However, we remark that there are Seifert fibered $3$-orbifolds with base $2$-orbifold different than a sphere with at most two cone points, whose underlying topological space is still a lens space; when this happens, the singular set does not entirely consist of a union of fibers. For instance, t}he other orbifolds in {the third case of Theorem \ref{main thm intro}} (whose base $2$-orbifold is a $2$-disk) can be obtained as a quotient of an  ``orbifold  lens space''  by an involution whose action is not free; in this case  the underlying topological space is always $S^3$, see also \cite{Dun2}.

The  case  $\mathcal{B}\cong S^2(2,2,b)$   contains the classical family of prism manifolds. Prism manifolds admit two inequivalent fibrations, the second one with   $\mathcal{B}\cong \mathbb{R}P^2(b)$, see \cite{orlik} {or \cite[Theorem 2.3]{hatchernotes}. We recover} an explicit description of the relations between the two fibrations of prism manifolds in Case 1 of Subsection \ref{subsec:finitely}. 

As a result of our analysis, we also obtain the following statement in analogy with the situation for spherical Seifert $3$-manifolds:
\begin{Theorem}
If a closed spherical Seifert fibered $3$-orbifold admits several inequivalent fibrations, then its underlying topological space is either a lens space or a prism manifold. 
\end{Theorem} 
However, unlike the manifold case, this is not a complete characterization of non-uniqueness, since there are $3$-orbifolds with underlying manifold a lens space, whose fibration is unique up to diffeomorphism.

Finally we remark that the presence of one  platonic group (tetrahedral, octahedral or icosahedral) or of one of their binary versions in the fundamental group of the $3$-orbifold assures, with {a few} exceptions, the uniqueness of the fibration.

The methods we use in the paper are related to the fact that  closed spherical $3$-orbifolds are \emph{globally} the quotient of the $3$-sphere $S^3$ by the action of a finite group $G$ of isometries. In \cite{mecchia-seppi, Mecchia-seppi2} we have analyzed different aspects of  the classification of finite subgroups of $\SO(4)$ up to conjugacy. Here we continue this kind of analysis with an additional difficulty consisting in considering a classification up to ``fibration-preserving conjugacy''.

\subsection*{Organization of the paper} In Section \ref{sec spherical 3-orbifolds}  we give an introduction to orbifolds of dimension $2$ and $3$, with special attention to the spherical case,  and we recall the definition of  Seifert fibration for orbifolds. In Section \ref{sec: algebraic} we discuss the  classification of finite subgroups of $\SO(4)$. {In Section \ref{sec: Seifert S3 invariant}} we analyze which groups leave invariant the Seifert fibrations of the 3-sphere and we explain {the approach} to the classification of finite subgroups of $\SO(4)$ {we adopt} to get a classification  of Seifert fibered  spherical $3$-orbifolds by diffeomorphism type.  In Section \ref{classification} we explicitly present the classification by distinguishing the case of $3$-orbifolds admitting finitely many inequivalent fibrations and the case of $3$-orbifolds with {infinitely many ones}.

\subsection*{Acknowledgements} We are very grateful to an anonymous referee for suggesting several valuable improvements on a previous version of this manuscript. The first author was partially supported by the FRA 2018 grant  ``Aspetti geometrici, topologici e computazionali delle variet\`{a}'', Universit\`{a} degli Studi di Trieste. The authors are members of the national research group GNSAGA

\clearpage

\section{Spherical and Seifert fibered three-orbifolds} \label{sec spherical 3-orbifolds}

\subsection{Spherical orbifolds}\label{sec:defi orbi}
Let us start by recalling some notions on smooth and spherical orbifolds in any dimension. For details see \cite{boileau-maillot-porti}, \cite{ratcliffe} or \cite{choi}. 

\begin{Definition}\label{Defi iniziale}
A \emph{smooth orbifold} $\mathcal{O}$ {\emph{(without boundary)}} of dimension $n$ is a paracomapct Hausdorff topological space $X$ endowed with an atlas $\varphi_i:U_i\to \widetilde U_i/\Gamma_i$, where:
\begin{itemize}
\item The $U_i$ form an open covering of $\mathcal O$.
\item The $\widetilde U_i$ are open subsets of $\R^n$ on which the finite groups $\Gamma_i$ act smoothly and effectively.
\item Each $\varphi_i$ is a homeomorphism and the compositions $\varphi_j\circ\varphi_i^{-1}$ lift to diffeomorphisms $\widetilde \varphi_{ij}:\widetilde U_i\to\widetilde U_j$. 
\end{itemize}
Moreover, the orbifold $\mathcal O$ is:
\begin{itemize}
\item \emph{Orientable} if the groups $\Gamma_i$ and the lifts $\widetilde \varphi_{ij}$ preserve an orientation of $\R^n$ (and the choice of such an orientation makes $\mathcal O$ oriented). 
\item \emph{Spherical} if each $\widetilde U_i$ is endowed with a Riemannian metric $\widetilde g_i$ of constant curvature $1$ preserved by the action of the groups $\Gamma_i$ and such that each $\widetilde \varphi_{ij}$ is an isometry from $(\widetilde U_i,\widetilde g_i)$ to $(\widetilde U_j,\widetilde g_j)$. 
\end{itemize}  
\end{Definition}
The topological space $X$ is called the \emph{underlying topological space} of the orbifold. 
A \emph{diffeomorphism}  between orbifolds $\mathcal O$ and $\mathcal O'$ is a homeomorphism $f:X\to X'$ of the underlying topological spaces such that each composition $\varphi'_{j'}\circ f|_{U_i}\circ \varphi_i^{-1}$, when defined, can be lifted to a diffeomorphism from $\widetilde U_i$ to its image in $\widetilde U'_{j'}$. If $\mathcal O$ and $\mathcal O'$ are oriented, then $f$ is \emph{orientation-preserving} if the lifts preserve the orientation of the $\widetilde U_i$ and $\widetilde U'_{j'}$. If $\mathcal O$ and $\mathcal O'$ are spherical, $f$ is an \emph{isometry} if the lifts are isometric for the Riemannian metrics $\widetilde g_i$ and $\widetilde g'_{j'}$.

One can define a \emph{local group} associated to every point $x$, namely the smallest possible group $\Gamma$ which gives a local chart $\varphi:U\to\widetilde U/\Gamma$ for $x$. If the local group is trivial, then $x$ is a  \emph{regular point} of $\mathcal O$. Otherwise $x$ is a \emph{singular point}. The set of regular points of $\mathcal O$ is a smooth manifold. {Manifolds are thus special cases of orbifolds, for which the group $\Gamma_i$ in Definition \ref{Defi iniziale} is always the trivial group. We do not assume that the singular locus is non-empty here, hence all the results of this paper also hold when $\mathcal O$ is a manifold.}

{One can give more generally the definition of smooth orbifold with boundary by replacing $\R^n$ by a half-space in $\R^n$ in Definition \ref{Defi iniziale}; we will not need such a notion in this paper, since we only consider closed orbifolds, as in the following definition.
\begin{Definition}
A \emph{closed orbifold} is a smooth orbifold (without boundary) whose underlying topological space is compact.
\end{Definition}
The underlying topological space of a closed orbifold might be a manifold with boundary (see Section \ref{subsec:orbifolds} below for examples), or also more pathological topological spaces having non-manifold points. 
However in Section \ref{subsec:orbifolds} below we will see that the underlying topological space of closed \emph{orientable} orbifolds of dimension three    is always a closed manifold and that of  closed  orbifolds of dimension two   is a compact  manifold with possibly non-empty boundary.}

The most intuitive examples of orbifolds are produced as quotients $\mathcal O=M/G$, for $G$ a group acting smoothly and properly discontinuously on a manifold $M$. In this case the local group of a point $x$ in the quotient $M/G$ is the stabiliser of any of the preimages of $x$ (which is finite). An orbifold is called \emph{good} if it is diffeomorphic to a quotient $M/G$ as above. Otherwise it is called \emph{bad}. 

\begin{Theorem}[{\cite[Theorem 13.3.10]{ratcliffe}}] \label{thm: compact spherical good}
Every closed spherical orbifold is good, and is in fact isometric to a global quotient $S^n/G$ for $G<\O(n+1)$ a finite group of isometries of $S^n$.
\end{Theorem} 
If the spherical orbifold is orientable, then $G$ is a subgroup of $\SO(n+1)$.  The following theorem of de Rham is a rigidity result for spherical orbifolds:

\begin{Theorem}[{\cite{derham,MR520507}}]  \label{derham thm}
If two closed spherical orbifolds  are diffeomorphic, then they are isometric. 
\end{Theorem}

If we consider two closed spherical orbifolds of the form (by Theorem \ref{thm: compact spherical good}) $\mathcal O=S^n/G$ and 
$\mathcal O'=S^n/G'$, then $\mathcal O$ and $\mathcal O'$ are isometric if and only if $G$ and $G'$ are conjugate in $\O(n+1)$. If moreover $\mathcal O$ and $\mathcal O'$ are orientable (i.e. $G,G'<\SO(n+1)$) and endowed with the orientation induced by the orientation of $S^n$, then $\mathcal O$ and $\mathcal O'$ have an orientation-preserving isometry if and only if $G$ and $G'$ are conjugate in $\SO(n+1)$.

Hence the classification of closed spherical orientable 3-orbifolds up to orientation-preserving diffeomorphisms amounts algebraically to the classification of finite subgroups of $\SO(4)$ up to conjugacy, which is the content of Section \ref{sec: algebraic}.




\subsection{Two and three-dimensional orbifolds}\label{subsec:orbifolds}
Let us start by considering orbifolds of dimension 2. The underlying topological space turns out to be a manifold with boundary. In fact a neighborhood of any point $x$ is modelled on $D^2/\Gamma$ where $\Gamma$ can be (see Figure~\ref{lm2o}):
\begin{itemize}
\item The trivial group, if $x$ is a regular point;
\item A cyclic group of rotations (in this case $x$ is called  \emph{cone point} and is labelled with the order of $\Gamma$);
\item A group of order 2 generated by a reflection ($x$ is called  \emph{mirror reflector} and is a boundary point of the underlying 2-manifold);
\item A dihedral group ($x$ is called \emph{corner reflector}, is still a boundary point for the underlying manifold and is labelled with the order of the rotation subgroup of $\Gamma$).
\end{itemize}

\begin{figure}[htbp]
\centering
\begin{minipage}[c]{.3\textwidth}
\centering
\includegraphics[width=.7\textwidth]{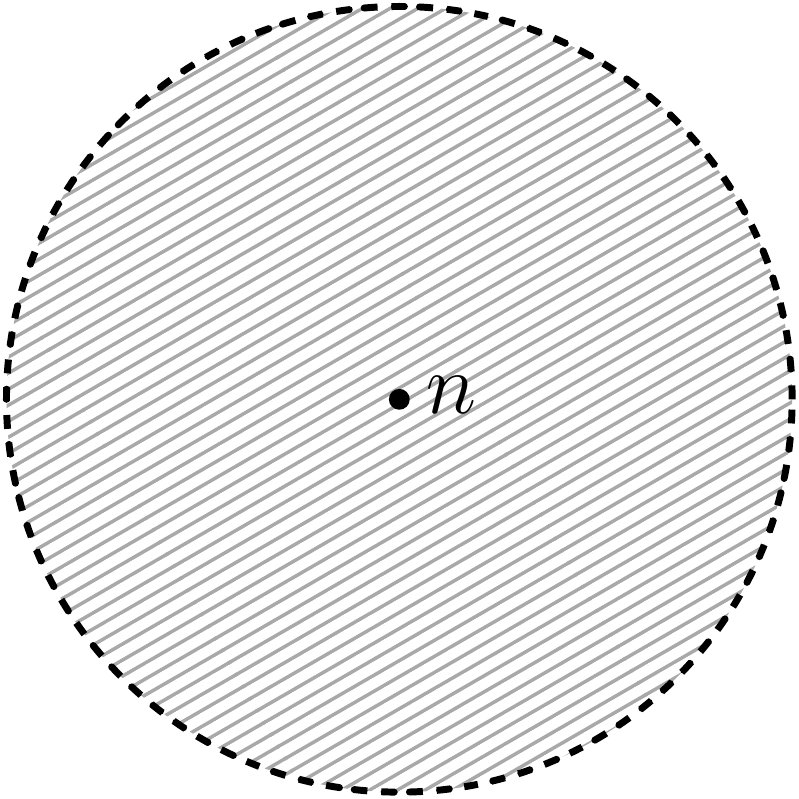} 
\end{minipage}%
\hspace{5mm}
\begin{minipage}[c]{.3\textwidth}
\centering
\includegraphics[width=0.8\textwidth]{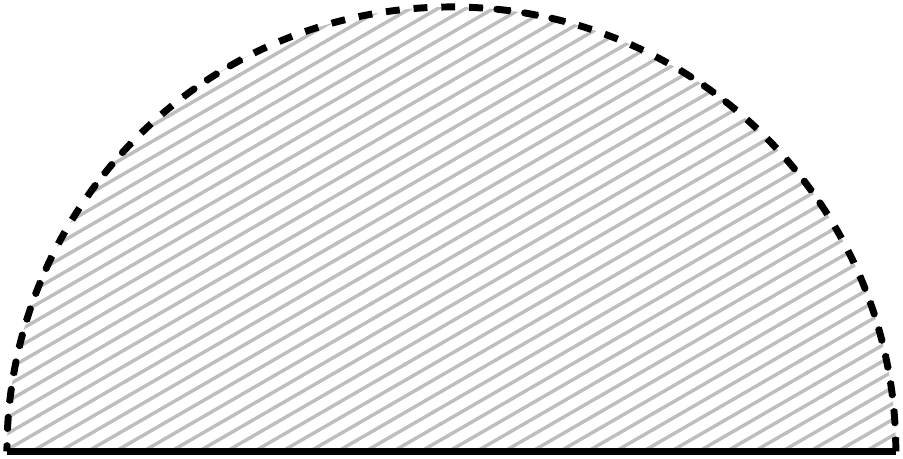} 
\end{minipage}%
\hspace{5mm}
\begin{minipage}[c]{.3\textwidth}
\centering
\includegraphics[width=.75\textwidth]{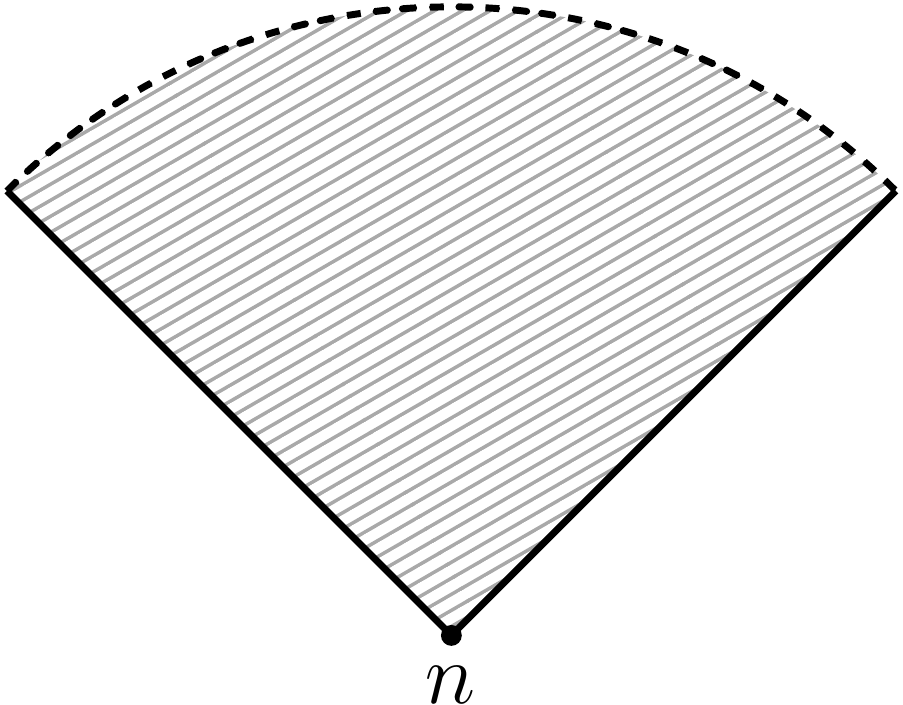} 
\end{minipage}
\caption{Local models of 2-orbifolds. From left to right, a cone point, a mirror reflector and a corner reflector.}\label{lm2o}
\end{figure}

If $\mathcal O$ is closed, its diffeomorphism type is denoted by 
 $X(n_1,\dots,n_k;m_1,\dots,m_h)$, where $X$ is the underlying manifold with boundary, $n_1,\ldots,n_k$ are the labels  of  cone points and $m_1,\ldots,m_h$ are the labels of  corner reflectors. (Labels are also called \emph{singularity indices}.) 

The \emph{Euler characteristic} of  $\mathcal O$ is then defined as:
 $$\chi(\mathcal O):=\chi(X)-\sum_{i}\left(1-\frac{1}{n_i}\right)-\frac{1}{2}\sum_{j}\left(1-\frac{1}{m_j}\right)~$$
where $\chi(X)$ is the Euler characteristic  of the underlying manifold $X.$

As a consequence of the discussion of Section \ref{sec:defi orbi}, any closed spherical 2-orbifold is diffeomorphic to a quotient of $S^2$ by a finite subgroup $G$ of $\O(3)$, and moreover the conjugacy class of $G$ determines both the diffeomorphism and isometry type of $S^2/G$. Starting with the orientable case, the following classical result classifies finite subgroups of $\SO(3)$:

\begin{Lemma} \label{lemma finite subgroups SO3}
A finite subgroup of $\SO(3)$ is either a cyclic group, a dihedral group, or the tetrahedral, octahedral or icosahedral group.
\end{Lemma}

It follows that closed orientable spherical 2-orbifolds are (up to diffeomorphism):
\begin{equation}\label{X>0 1}
 S^2, S^2(p,p), S^2(2,2,p), S^2(2,3,3), S^2(2,3,4), S^2(2,3,5) \quad\text{for }p\geq 2~.
 \end{equation} 
 For non-orientable spherical orbifolds, it suffices to consider order 2 quotients of orientable ones, thus getting the list: 
 \begin{gather}
  D^2, D^2(p;), D^2(;p,p), D^2(2;p), D^2(;2,2,p), D^2(3;2), D^2(;2,3,3), D^2(;2,3,4), D^2(;2,3,5),\nonumber \notag \\
  \R P^2, \R P^2(p) \quad\text{for }p\geq 2~. \label{X>0 2}
  \end{gather} 
These are in fact all good two-dimensional orbifolds of positive Euler characteristic. The additional bad orbifolds with $\chi(\mathcal O)>0$ are:
\begin{equation}\label{X>0 3}
S^2(p,q) \text{ and } D^2(p,q) \quad\text{for }p\neq q~.
\end{equation} 

\medskip

Let us now move on to dimension three. In this paper we only consider orientable 3-orbifolds. By a standard argument, any point $x$ admits a local chart of the form $D^3/\Gamma$ for $\Gamma$ a finite subgroup of $\SO(3)$, and the local model is thus the cone over one of the spherical orientable 2-orbifolds listed above, listed in \eqref{X>0 1}. It follows that the underlying topological space  is a manifold and the singular set is a trivalent graph; the local group is cyclic in the complement of the vertices of the graph, and the edges are thus labelled with a singularity index which is the order of the cyclic group.

\begin{figure}[htbp]
\centering
\begin{minipage}[c]{.2\textwidth}
\centering
\includegraphics[height=3cm]{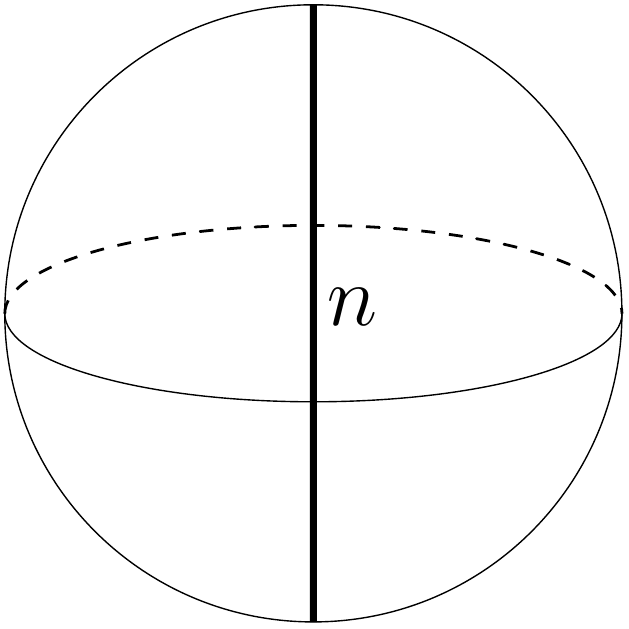} 
\end{minipage}%
\begin{minipage}[c]{.2\textwidth}
\centering
\includegraphics[height=3cm]{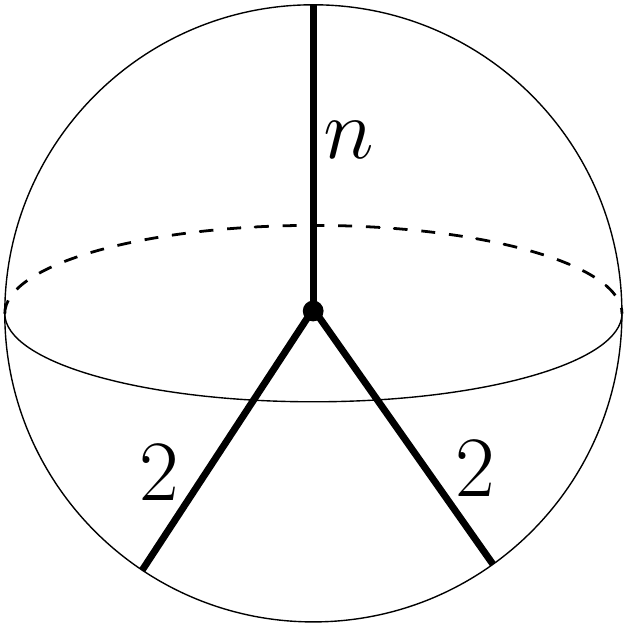} 
\end{minipage}%
\begin{minipage}[c]{.2\textwidth}
\centering
\includegraphics[height=3cm]{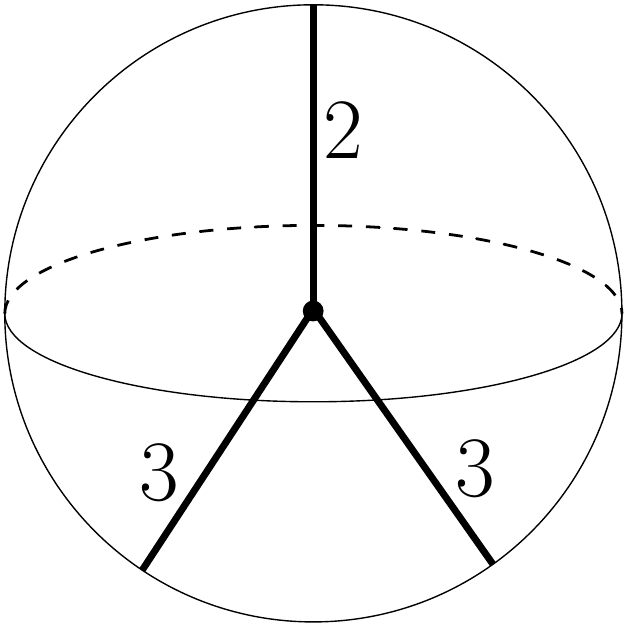} 
\end{minipage}%
\begin{minipage}[c]{.2\textwidth}
\centering
\includegraphics[height=3cm]{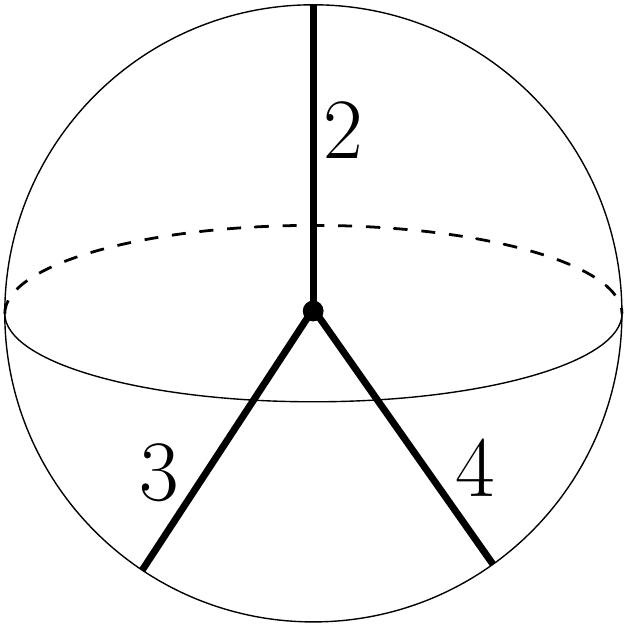} 
\end{minipage}%
\begin{minipage}[c]{.2\textwidth}
\centering
\includegraphics[height=3cm]{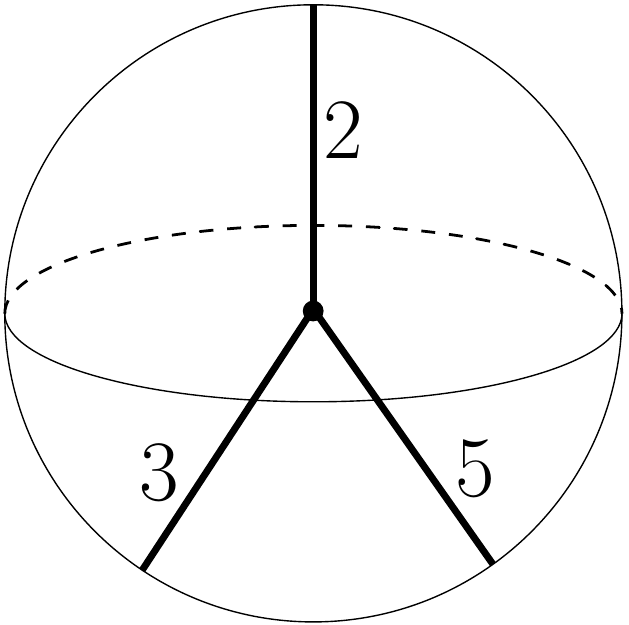} 
\end{minipage}%
\caption{Local models of orientable 3-orbifolds.}\label{lm3o}
\end{figure}



\subsection{Seifert fibrations on three-dimensional orbifolds} \label{sec fibrations}
Let us turn our attention to a topological description of the quotient in terms of Seifert fibrations  for orbifolds, which we now define.

\begin{Definition} \label{defi seifert fibration}
Given a three-dimensional orientable orbifold $\mathcal O$, a Seifert fibration is a surjective map $\pi:\mathcal O \rightarrow \mathcal B$ with image a two-dimensional orbifold  $\mathcal B$, such that for every point $x\in\mathcal  B$ there exist:
\begin{itemize}
\item An orbifold chart $\varphi:U\cong \widetilde U/\Gamma$ for $\mathcal B$ around $x$;
\item An action of $\Gamma$ on $S^1$;
\item An orbifold diffeomorphism $\phi:\pi^{-1}(U)\to (\widetilde U\times S^1)/\Gamma$, where $\Gamma$ acts diagonally on $\widetilde U\times S^1$ by preserving the orientation;
\end{itemize}
such that the following diagram

\[
\xymatrix{
\pi^{-1}(U) \ar[d]_-{\pi} \ar[r]^-{\phi} & (\widetilde{U}\times S^1)/ \Gamma  \ar[d] & \widetilde{U}\times S^1 \ar[l] \ar[d] \\
  U \ar[r]^-{\varphi} & \widetilde{U}/ \Gamma  &  \widetilde{U}  \ar[l]
  }
\] 
is commutative, with the obvious maps on  unspecified arrows.
\end{Definition}

Since the action of $\Gamma$ is required to preserve the orientation (as a consequence of the assumption that $\mathcal O$ is supposed orientable) each element of $\Gamma$ may  either preserve both the orientation of $\widetilde U$ and $S^1$, or reverse both orientations.
Observe that each fiber $\pi^{-1}(x)$ is topologically either a simple closed curve or an interval. A fiber which projects to a regular point of $\mathcal B$ is called \emph{generic}; it is called \emph{exceptional} otherwise.

Let us now consider the local models for oriented Seifert fibered orbifolds. More details can be found in  \cite{bonahon-siebenmann} or \cite{dunbar2}.

\begin{itemize}
\item If the fiber $\pi^{-1}(x)$ is generic, one can pick $\Gamma$ the trivial group in Definition \ref{defi seifert fibration}, hence $\pi^{-1}(x)$ has a tubular neighborhood with a trivial fibration. 
\item
If $x$ is a cone point labelled by $b$, the local group $\Gamma$ is a cyclic group of order $b$ acting by rotations on $\widetilde U$ and thus it needs to act on $S^1$ by rotations too. Hence $\pi^{-1}(x)$ has a fibered neighborhood which is a solid torus, fibered in the usual sense of Seifert fibrations for manifolds, except that the central fiber might be singular. The \emph{local invariant} of $\pi^{-1}(x)$ is defined as the ratio $a/b\in\mathbb{Q}/\mathbb{Z}$, where a generator of $\Gamma$ acts on $\widetilde U$ by rotation of an angle ${2\pi}/{b}$ and on $S^1$ by rotation of $-{2\pi a}/{b}$. Up to adding integer multiples of $b$, one can in fact choose $a$ so that $a/b\in[0,1)$.
The  index of singularity $\pi^{-1}(x)$ is thus $\gcd(a,b)$ (meaning that  points with index of singularity 1 are regular). 
{See Figure \ref{fig:seifert case 1}.}

{It is worth remarking that three qualitatively different situations may occur here:
\begin{enumerate}
\item If $a\equiv 0\,\mod\,b$, then $\pi^{-1}(x)$ is a circle in the singular set of the orbifold $\mathcal O$, with local group $\mathbb{Z}_b$ associated to each point. Forgetting the singular locus, in the underlying manifold $\pi^{-1}(x)$ has a tubular neighborhood endowed with a trivial fibration.
\item If $a$ and $b$ are relatively prime, then the action of the cyclic group is free: in this case $\pi^{-1}(x)$ has a neighborhood consisting only of regular points, fibered (non-trivially) in the usual  sense of Seifert fibrations for manifolds.
\item Finally, if $a$ and $b$  are not relatively prime and $a\not\equiv 0\,\mod\,b$, then $\pi^{-1}(x)$ is contained in the singular set, with index $\mathrm{gcd}(a,b)\neq 1$, and at the same time its tubular neighborhood is non-trivially fibered at the level of underlying Seifert manifold.
\end{enumerate}
}

\begin{figure}[htbp]
\centering
\begin{minipage}[c]{.5\textwidth}
\centering
\includegraphics[width=.7\textwidth]{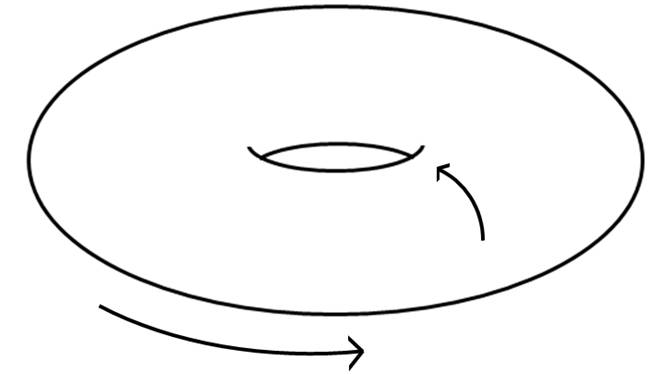} 
\end{minipage}%
\begin{minipage}[c]{.5\textwidth}
\centering
\includegraphics[width=.9\textwidth]{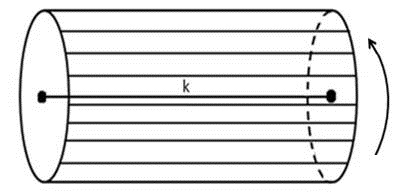} 
\end{minipage}%
\caption{The preimage of a cone point in the base orbifold. On the left, the cyclic action on a solid torus $D^2\times S^1$ generated by a simultaneous rotation of angle ${2\pi}/{b}$ on $D^2$ and of angle $-{2\pi a}/{b}$ on $S^1$. On the right, the quotient is identified to a solid cylinder (a fundamental domain for the previous action) where top and bottom are glued by a rotation, and the central fibre has singularity index $k=\mathrm{gcd}(a,b)$.}\label{fig:seifert case 1}
\end{figure}

\item
If $x$ is a mirror reflector, thus with local group $\mathbb{Z}_2$ whose generator acts by reflection on $\widetilde U$ and thus also on $S^1$, the local model is topologically a 3-ball. The fiber $\pi^{-1}(x)$ is an interval. The endpoints of all the fibers $\pi^{-1}(x)$, as $x$ varies in the mirror reflector, form two disjoint singular arcs of index 2.  {See Figure \ref{fig:seifert case 2} with $k=1$.}

\item
If $x$ is a corner reflector, namely $\Gamma$ is a dihedral group, by a similar argument the non-central involutions in $\Gamma$ act by simultaneous reflection both on $\widetilde U$ and on $S^1$. The local model is again a topological 3-ball (called solid pillow) with some singular set inside. The fiber  $\pi^{-1}(x)$ is again an interval and might be singular in this case, while the preimages of the nearby mirror reflectors are intervals as in the above case. The local invariant associated to $\pi^{-1}(x)$ is defined as the local invariant of the cyclic index two subgroup, and the singularity index is $\gcd(a,b)$, {see Figure \ref{fig:seifert case 2} again.}
\end{itemize}

\begin{figure}[htbp]
\centering
\begin{minipage}[c]{.6\textwidth}
\centering
\includegraphics[width=.8\textwidth]{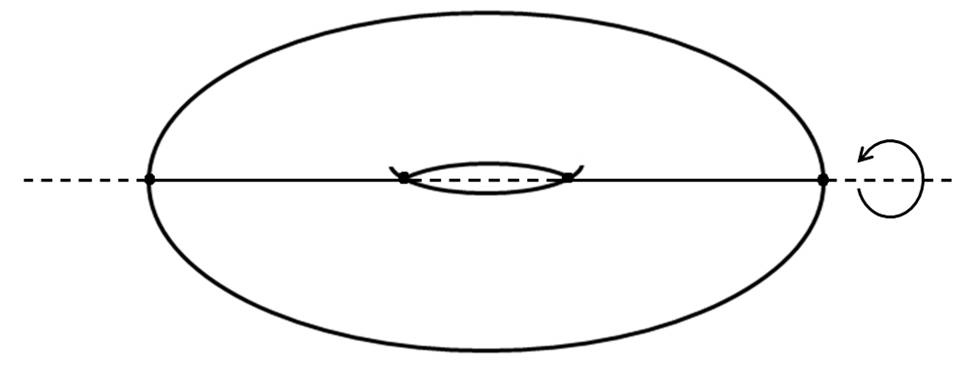} 
\end{minipage}%
\begin{minipage}[c]{.4\textwidth}
\centering
\includegraphics[width=.65\textwidth]{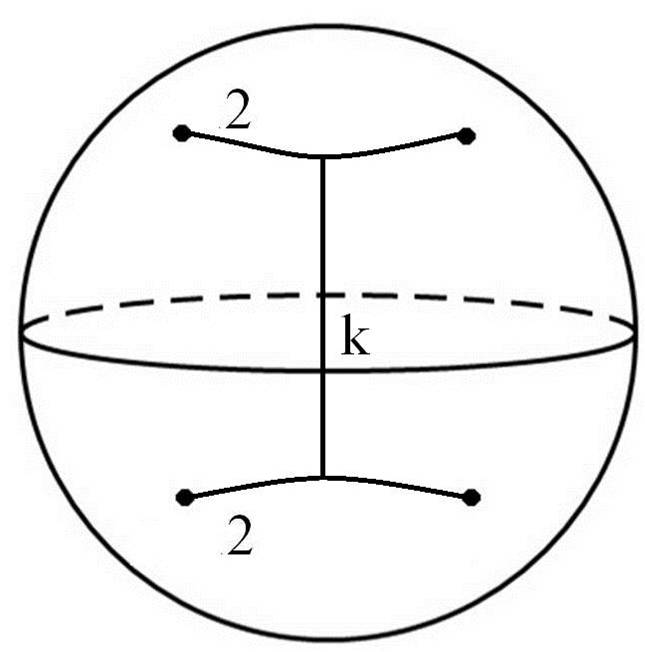} 
\end{minipage}%
\caption{On the left, the involution acting on a solid torus by a simultaneous reflection. On the right, a solid pillow, namely the preimage of a neigbourhood of a mirror reflector ($k=1$) or a corner reflector ($k=\gcd(a,b)$) in the base orbifold.}\label{fig:seifert case 2}
\end{figure}




Together with the base orbifold and the local invariants, to a Seifert fibered orbifold is associated the Euler class $e$ and an additional invariant $\xi\in\{0,1\}$ for each boundary component of the underlying manifold of the base orbifold. These invariants satisfy the relation:

\begin{equation}\label{somma eulero invarianti}  
e+\sum_i\frac{a_i}{b_i}+\frac{1}{2}\left(\sum_j\frac{a'_{j}}{b'_{j}}+\sum_k\xi_k\right)\equiv 0\;\textnormal{mod}\,1~,
\end{equation}
where the first sum involves local invariants associated to all cone points, the second sum is taken over mirror reflectors, and the third sum is taken over boundary components. 

The base orbifold together the local  invariants, the Euler class and the boundary components invariants  determine the Seifert fibered  orbifold up to  orientation-preserving and fibration-preserving  diffeomorphism.

\begin{Remark} \label{rmk:change orientation}
Let us observe that changing the orientation of the orbifold inverts the signs of local invariants and  of the Euler number.
\end{Remark}

{To denote a specific fibered Seifert 3-orbifold we use  the following compact notation:
\begin{itemize}
\item If the base orbifold is a 2-manifold or a 2-orbifold admitting only cone points,   $ (X;a_1/b_1,\dots,a_n/b_n; e)$ is the Seifert fibered 3-orbifold with base orbifold $X$,   local invariants of the cone points $a_i/b_i$ and Euler number $e$. When $X$ is a manifold we write  simply $(X;\, ; e)$
\item If the base orbifold admits  also mirror reflectors and possibly  corner points, then the Seifert fibered 3-orbifold $ (X;a_1/b_1,\dots,a_n/b_n;a'_1/b'_1,\dots,a'_m/b'_m; e; \xi_1,\dots\xi_a)$ has  base orbifold $X$,  local invariants of  cone points $a_i/b_i$,  local invariants of  corner  points $a'_j/b'_j$, Euler number $e$ and boundary components invariants $\xi_k$. If the singular set of $X$ does not contain  any cone point (resp. corner reflector)  we write $ (X;\, ;a'_1/b'_1,\dots,a'_m/b'_m; e; \xi_1,\dots\xi_a)$, (resp.  $(X;a_1/b_1,\dots,a_n/b_n;\,; e; \xi_1,\dots\xi_a)$)
\end{itemize}}

\subsection{Comparison with the classification of Seifert manifolds} \label{rmk manifold}


As already observed, if $\mathcal O$ is a manifold (meaning that the singular set is empty), then the notion of Seifert fibration according to Definition \ref{defi seifert fibration} coincides with the usual notion of Seifert fibration for manifolds. If $\mathcal O$ is not a manifold but the base orbifold $\mathcal B$ is orientable, which means that the only singular points of $\mathcal B$ are cone points, then the underlying topological space of $\mathcal O$ is a Seifert fibered manifold, whose invariants in the sense of the previous section are simply obtained by ``simplifying the fractions'' $a_i/b_i$. Namely, by replacing the invariants $a_i$ and $b_i$ of each cone point in the base orbifold by $a_i/\gcd(a_i,b_i)$ and $b_i/\gcd(a_i,b_i)$ respectively.

However, the classifying data for Seifert fibrations of manifolds are more standardly defined as the following data:
\begin{equation}\label{eq:data seifert fibration classical}
(X;\alpha_1/\beta_1,\ldots,\alpha_n/\beta_n)~,
\end{equation}
where $X$ is a closed surface (possibly non-orientable) and the pairs $(\alpha_i,\beta_i)$ are relatively prime. Namely, one assumes that the fractions $\alpha_i/\beta_i$ cannot be simplified, but one does \emph{not} take the class of $\alpha_i/\beta_i$ in $\Q/\Z$, or in other words, one cannot assume that $\alpha_i/\beta_i$ is in $[0,1)$ as we did for $a_i/b_i$ in the previous section. The representation of a Seifert fibered manifold by the data as in \eqref{eq:data seifert fibration classical} is not unique: two representations of the same Seifert fibered manifold differ by replacing each $\alpha_i/\beta_i$ by $\alpha_i'/\beta_i'=\alpha_i/\beta_i+k_i$ for $k_i\in\mathbb Z$, under the constraint that $\sum_i k_i=0$. It is not allowed to have $n=0$, that is, no marked point in $X$: in this case one needs to introduce an invariant $\alpha/\beta$ with $\beta=1$. This is necessary in order to determine the Euler class $e$, see below. For instance, this allows to distinguish  between $S^3$ and $S^2\times S^1$ which both have base surface $X=S^2$.

For a Seifert fibered manifold $\mathcal M$, let us now briefly explain how to go from the invariants introduced in Section \ref{sec fibrations} (by interpreting $\mathcal M$ as a Seifert fibered orbifold $\mathcal O$ with empty singular set) to the classical invariants of \eqref{eq:data seifert fibration classical}, and viceversa.

Given \emph{a} representation of $\mathcal M$ as in \eqref{eq:data seifert fibration classical}, the base orbifold $\mathcal B$ has underlying topological space $X$ and $n$ cone points with indices $\beta_1,\ldots,\beta_n$. The local invariants are the classes of the ratios $\alpha_i/\beta_i$ in $\Q/\Z$. Finally, the Euler class equals $e=-\sum_i(\alpha_i/\beta_i)$. Observe that these data do not change if one changes the representation of $\mathcal M$ as explained above.

Conversely, given the Seifert invariants in the sense of orbifolds, clearly $X$ is the topological surface underlying the base orbifold $\mathcal B$. To determine the ratios $\alpha_i/\beta_i\in\Q$, it suffices to pick some representatives of the classes of $a_i/b_i$ in $\Q/\Z$ so that $-\sum_i(\alpha_i/\beta_i)$ coincides with the Euler class.

\subsection{The spherical case} \label{sec seifert s3}

If a Seifert fibered orbifold $\mathcal O$ is geometric, i.e. it admits a metric locally modelled on one of Thurston's eight geometry, then its geometry is detected by the Euler charactersitic of the base orbifold and by the Euler number of the fibration (see \cite[page 71]{Dun2}). In particular for the spherical case we have:
\begin{Proposition} \label{prop determine geometry}
Let $\mathcal O$ be a closed spherical orbifold  and $\pi:\mathcal O\to\mathcal B$ be a Seifert fibration. Then 
\begin{equation}\label{determine geometry}
\chi(\mathcal B)>0\quad\text{ and }\quad e(\pi)\neq 0~.
\end{equation} 
Conversely, every closed Seifert fibered orbifold satisfying the conditions in \eqref{determine geometry} is spherical.
\end{Proposition}
For the last part of the statement, see \cite[Proposition 2.13]{boileau-maillot-porti} and its proof. 

 {In \cite{bonahon-siebenmann} the definition of Seifert fibration for orbifolds is given for orbifolds which might be non-orientable, hence in a more general setting with respect to our Definition \ref{defi seifert fibration}. We omit the complete definition in the non-orientable case in this paper. However, the following statement, which is a consequence of Proposition \ref{prop determine geometry}, shows that in the case of spherical orbifolds we can harmlessly reduce to the case of orientable orbifolds.
 \begin{Corollary}\label{cor nonoriented}
 Every closed spherical Seifert fibered orbifold is orientable.
 \end{Corollary}
 \begin{proof}
Let $\mathcal O$ be a non-orientable orbifold which is Seifert fibered in the sense of \cite{bonahon-siebenmann}. It was proved in \cite{bonahon-siebenmann} that the Seifert fibration can be lifted to its orientation double cover $\widetilde{\mathcal O}$, which is therefore endowed with a Seifert fibration $\pi$. We claim that the Euler class of $\pi$ vanishes. To see this, let $\varphi$ be the generator of the deck transformation group of the orbifold covering $\widetilde{\mathcal O}\to \mathcal O$. That is, $\varphi$ is an (order-two) orientation-reversing self-diffeomorphism of $\widetilde{\mathcal O}$. In other words, if we choose an orientaton $\mathfrak o$ on $\widetilde{\mathcal O}$ and we denote by $\mathfrak o'$ the opposite orientation, then $\varphi:(\widetilde{\mathcal O},\mathfrak o)\to (\widetilde{\mathcal O},\mathfrak o')$ is an orientation-preserving orbifold diffeomorphism. Moreover $\varphi$ preserves the Seifert fibration $\pi$ by construction. Since the Euler class is an invariant of oriented Seifert fibered orbifolds, we have $e_{\mathfrak o}(\pi)=e_{\mathfrak o'}(\pi)$. (Of course $e_{\mathfrak o}(\pi)$ denotes the Euler class of $\pi$ with respect to the orientation $\mathfrak o$.) But by Remark \ref{rmk:change orientation}, $e_{\mathfrak o}(\pi)=-e_{\mathfrak o'}(\pi)$. This implies that 
the Euler class of $\pi$ vanishes. Therefore a Seifert fibered non-orientable orbifold $\mathcal O$ cannot be spherical, for otherwise $\widetilde{\mathcal O}$ would also inherit a spherical structure, and by Proposition \ref{prop determine geometry} the Euler class of $\pi$ should be different from zero.
\end{proof}}

{In the proof of Corollary \ref{cor nonoriented} we have used that, given an orbifold covering $\mathcal O\to\mathcal O'$, any Seifert fibration for $\mathcal O'$ can be lifted to $\mathcal O$. This fact has a second important consequence.
 Namely, in light of Theorem \ref{thm: compact spherical good}, we can reduce our analysis to the study of the fibrations induced on the quotient $S^3/G$ by the Seifert fibrations of $S^3$.}
 Recalling again that, if $\mathcal O$ is a manifold, then Definition \ref{defi seifert fibration} coincides with the usual Seifert fibrations for manifolds,
 Seifert fibrations for $S^3$ are well known (see \cite{seifert} or also \cite[Proposition 5.2]{geigeslange}): they have 
 base orbifold $S^2(u,v)$ for $u,v\geq 1$ two coprime integers, hence only two non-generic fibers  which are non-singular. The local invariants are given by (the classes modulo 1 of)
 $\bar v/u$ and $\bar u/v$ where $u\bar u+v \bar v=1$,  and the Euler class is $\pm 1/uv$.
 
 If $u$ or $v$ equals $1$, we mean that the correponding point in the base orbifold is regular, and hence there is no local invariant to associate (the above formula would indeed give 0 as output). In particular, for $u=v=1$ we obtain the \emph{Hopf fibration}, which is a fiber bundle in the usual sense, since every fiber has a tubular neighborhood which is fibered as a usual product.
  A more concrete description of these fibrations in terms of the geometry of $S^3$ will be provided in Section \ref{sec: Seifert S3 invariant}.

{\subsection{An example of non-uniqueness}\label{subsec:example nonuniqueness}
Let us briefly discuss a concrete example of a spherical orbifold with non-empty singular set which admits multiple fibrations. Let $\mathcal O$ be the orbifold having underlying topological space the 3-sphere and singular set the Hopf link with singularity index $2$ for each component. Then $\mathcal O$ is easily seen to admit a Seifert fibration which is simply the Hopf fibration of the underlying $3$-sphere, and the singular locus consists of two fibers of the Hopf fibration. Such a fibration has base orbifold $S^2(2,2)$, local invariants $0/2$ over each cone point of the base orbifold and Euler class $e=-1$. By Proposition \ref{prop determine geometry}, $\mathcal O$ is spherical. For an illustration of the Hopf fibration, see for instance \cite[Chapter8, Figure 16]{zbMATH05080035} or \cite{hopf1,hopf2}.
}

{One can describe an alternative Seifert fibration of $\mathcal O$ as follows. We consider an annulus, with boundary equal to the union of  two singular circles of index 2,   fibered by intervals and a solid torus containing such annulus  fibered as in Figure \ref{fig:example}, where the fibers of the complement of the  annulus are meridians of the torus. Then we perform a Dehn twist on such solid torus with singular circles and we glue along its boundary another solid torus, identifying the meridian of the former to a longitude of the latter. Taking into account the two singular circles in the first torus, the union of the two tori give $\mathcal O$; moreover if we equip the second torus with the  trivial fibration, the fibrations of the two tori give a Seifert fibration of $\O$ different from that described in the first paragraph.  Such a  Seifert fibration has base orbifold $D^2$ (with no cone and corner points), Euler class $e=-1$, and invariant $\xi=0$ associated to the only boundary component of $D^2$.}

This example is a special case of three-orbifolds whose underlying topological space is $S^3$ and the singular locus is a 2-bridge link with local group of order two. The possible fibrations of such orbifolds are briefly described in Example \ref{ex:2bridge} at the end of the paper.

\begin{figure}[htbp]
\centering
\begin{minipage}[c]{.5\textwidth}
\centering
\includegraphics[width=.8\textwidth]{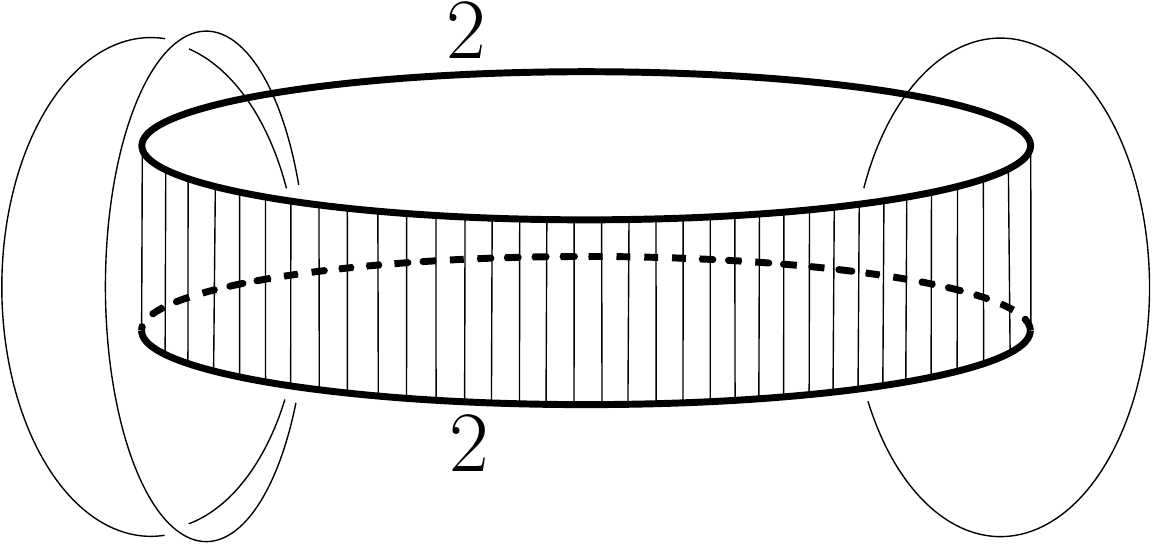} 
\end{minipage}%
\begin{minipage}[c]{.5\textwidth}
\centering
\includegraphics[width=.68\textwidth]{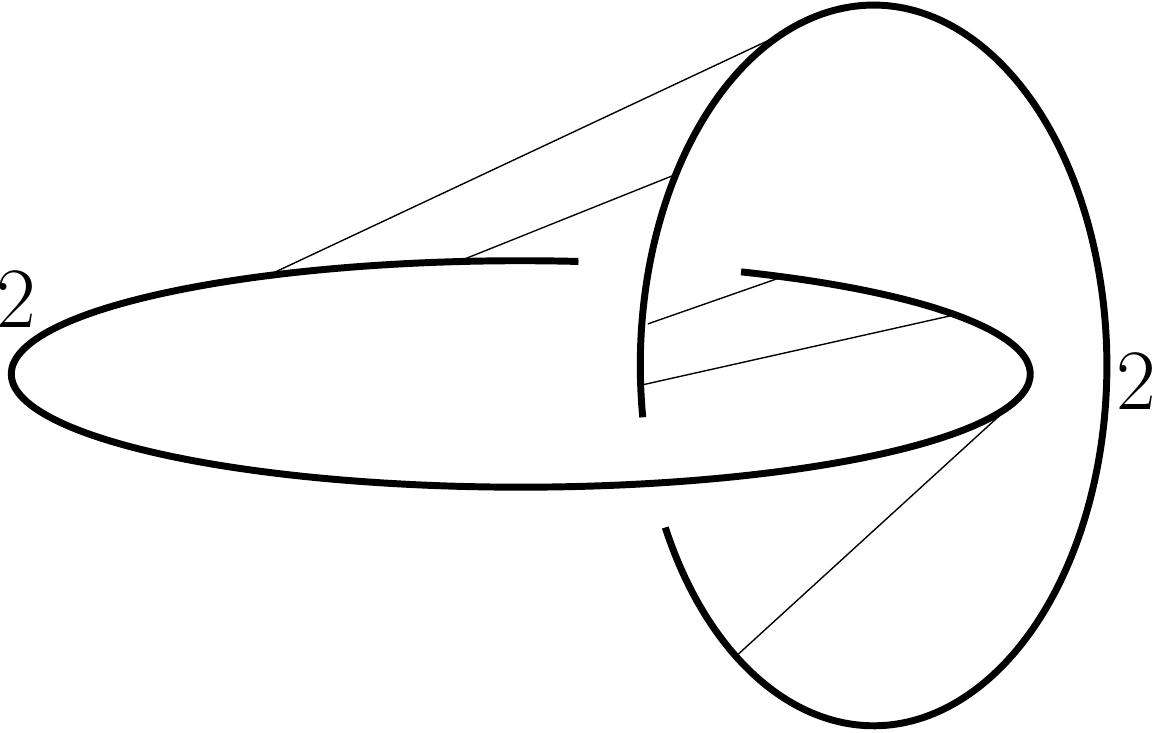} 
\end{minipage}%
\caption{On the left, a Seifert fibration of a solid torus with two singular circles in the interior: an annulus connecting the two singular circles is fibered by intervals, while the complement of such annulus in the solid torus is fibered by meridional nonsingular circles. Performing a Dehn twist on such solid torus, one gets a fibered solid neighbourhood of the Hopf link (on the right). Its complement in the underlying manifold $S^3$ is a trivially fibered solid torus.}\label{fig:example}
\end{figure}

\section{Finite subgroups of $\SO(4)$} \label{sec: algebraic}

In this section we discuss a classification of closed spherical orientable three-orbifolds up to orientation-preserving diffeomorphisms, from a group-theoretical point of view. In Section \ref{sec: Seifert S3 invariant} we will then focus on a classification up to fibration-preserving diffeomorphisms. 
\subsection{Quaternion algebra and subgroups of $S^3$}

As discussed in Section \ref{sec:defi orbi}, two spherical 3-orbifolds $\mathcal O=S^3/G$ and $\mathcal O'=S^3/G'$ are diffeomorphic if and only if they are isometric, and such an isometry can be lifted to an isometry of $S^3$ which conjugates $G$ to $G'$. If the isometry between the orbifolds is orientation-preserving, then the lift to $S^3$ is orientation-preserving. For this reason, the classification of closed spherical orientable three-orbifolds $S^3/G$ up to orientation-preserving diffeomorphisms corresponds to the algebraic classification of finite subgroups of $\SO(4)$ up to conjugation in $\SO(4)$, originally due to Seifert and Threlfall (\cite{threlfall-seifert} and \cite{threlfall-seifert2}), which we shall now briefly recall. For more details, see \cite{duval}, which we essentially follow although it must be mentioned that in Du Val's list of  finite subgroups of $\SO(4)$ there are three missing cases, see also \cite{conway-smith,mecchia-seppi,Mecchia-seppi2}.

Let  us identify $\R^4$ with the quaternion algebra $\H=\{a+bi+cj+dk\,|\,a,b,c,d\in\R\}=\{z_1+z_2j\,|\,z_1,z_2\in\C\}$. Given $q=z_1+z_2 j\in\H$, its conjugate is $\bar q=\bar z_1-z_2 j$. Thus the standard positive definite quadratic form of $\R^4$ is identified to $q\bar q=|z_1|^2+|z_2|^2$.
The three-sphere $S^3$ then corresponds to the set of unit quaternions:
\begin{equation*}
S^3=\{a+bi+cj+dk \,|\, a^2+b^2+c^2+d^2=1\}=\{z_1+z_2j\, \,|\, |z_1|^2+|z_2|^2=1\}\,.
\end{equation*}
which is thus endowed with a multiplicative group structure induced from that of $\H$. 



Let us now consider the group homomorphism $\Phi:S^3\times S^3\rightarrow \SO(4)$ which associates to $(p,q)\in S^3\times S^3$ the map $\Phi_{p,q}:\H \rightarrow \H$ defined by 
$\Phi_{p,q}(h)=phq^{-1}$, which is an isometry of $S^3$. It turns out that $\Phi$ is surjective with kernel $\mathrm{Ker}(\Phi)=\{\pm(1,1)\}$, hence $\Phi$ induces a  1-1 correspondence between  finite subgroups of $\SO(4)$ and finite subgroups of $S^3\times S^3$ containing $\{\pm(1,1)\}$.
Since $(-1,-1)$ is central, two subgroups are conjugate in $\SO(4)$ if and only if their preimages  are conjugate in $S^3\times S^3$.   
To give a classification  of finite subgroups $G$ of $\SO(4)$ up to conjugation, it thus suffices to classify the subgroups $\widehat G=\Phi^{-1}(G)<S^3\times S^3$ containing $\{\pm(1,1)\}$, up to conjugation in $S^3\times S^3$. 
The latter are uniquely determined by the 5-tuple $(L,L_K,R,R_K,\phi)$, where (if $\pi_i:S^3\times S^3\to S^3$ denotes the projection to the $i$-th factor):
\begin{multicols}{2}
\begin{itemize}
\item $L=\pi_1(\widehat G)$;
\item $L_K=\pi_1((S^3\times\{1\})\cap \widehat G)$;
\item $R=\pi_2(\widehat G)$;
\item $R_K=\pi_2((\{1\}\times S^3)\cap \widehat G)$;
\end{itemize}
\end{multicols}
\vspace{-0.6cm}
\begin{itemize}
\item $\phi:L/L_K\to R/R_K$ is a group isomorphism obtained by composing the isomorphisms of $\widehat G/(L_K\times R_K)$ with $L/L_K$ and with $R/R_K$, induced by $\pi_1$ and $\pi_2$ respectively.
\end{itemize}
 Based on \cite[Proposition 1]{Mecchia-seppi2}, two such 5-tuples $(L,L_K,R,R_K,\phi)$ and $(L',L'_K,R',R'_K,\phi')$ correspond to conjugate subgroups if and only if there exist $p,q\in S^3$ such that conjugation by $p$ (which we denote by $c_p$) maps $L$ to $L'$ and $L_K$ to $L_K'$, $c_q$ maps $R$ to $R'$ and $R_K$ to $R_K'$, and $\phi'=c_q\circ \phi\circ c_p^{-1}$.
  
It only remains to determine finite subgroups of $S^3$ up to conjugation. For this purpose, observe that the equatorial $S^2=\{bi+cj+dk \,|\, b^2+c^2+d^2=1\}$ in $S^3$ consists of points equidistant from the North and South poles, which are identified to $1$ and $-1$ in $\H$. Hence an element $(p,q)\in S^3\times S^3$ preserves $S^2$ if and only if $p=q$, and this gives a 2-to-1 group epimorphism
$S^3\to \SO(3)$
by composing the diagonal inclusion of $S^3$ in $S^3\times S^3$ with $\Phi$. Together with Lemma \ref{lemma finite subgroups SO3}, one obtains that the finite subgroups of $S^3$ up to conjugation are:

\begin{itemize}
\item The cyclic group $C_n= \{\cos\left({2\alpha\pi}/{n}\right)+i\sin\left({2\alpha\pi}/{n}\right)\,,\,\alpha=0,\ldots,n-1\}$ for $n\geq 1$. Observe that $C_n$ contains the center $-1$ if and only if $n$ is even.
\item The binary dihedral group $D^*_{4n}=C_{2n}\cup C_{2n} j$ for  $n\geq 2$, which is a central extension of the dihedral group $D_{2n}$ by a group of order 2.
\item The binary tetrahedral group $T^*=\cup_{r=0}^2 ({1}/{2}+{i}/{2}+{j}{2}+{k}/{2})^r D^*_{4}$, which is a central extension of the tetrahedral group.
\item The binary octahedral group $O^*=T^*\cup (\sqrt{2}/{2}+\sqrt{2}j/2) T^*$ which is a central extension of the octahedral group.
\item The binary icosahedral group $I^*=\cup_{r=0}^4 \left(\tau^{-1}/2+\tau j/2+ k/2\right)^r T^*$ which is a central extension of the icosahedral group, where $\tau=(\sqrt{5}+1)/{2})$.
\end{itemize}


Observe that for $n=1$, $D_{4}^ *=\{\pm 1,\pm j\}$ is conjugate to $C_4=\{\pm 1,\pm i\}$. For this reason, the groups $D_{4n}^*$ are only taken with indices $n\geq 2$. The case $n=2$, namely $D_8^ *=\{\pm 1,\pm i,\pm j,\pm k\}$, is also called quaternion group. 

We remark that thoughout the paper, we use the notation $C_n$ for the cyclic subgroup of $S^3$ defined in the first point of the list above, while we reserve the symbol $Z_n$ for the abstract cyclic group of $n$ elements.



\subsection{Listing finite subgroups of $\SO(4)$} \label{subsec finite subgroups}
The list of finite subgroups of $\SO(4)$ up to conjugacy is given, applying the above procedure and following \cite{duval}, in Table \ref{subgroup}, by means of the data determining their preimage $\widehat G=\Phi^{-1}(G)$. In most cases, the isomorphism $\phi$ is uniquely determined up to conjugacy, hence the last entry of the 5-tuple $(L,L_K,R,R_K,\phi)$ is omitted. In Families 1 and $1'$ (resp. 11 and $11'$), $\phi$ is a isomorphism between cyclic (resp. dihedral) groups of order $r$ (resp. $2r$), thus it is encoded by a subscript $s$ such that $\mathrm{gcd}(s,r)=1$, meaning that the canonical generator of $\Z_r$ is sent to $s$ times the canonical generator (and the non-central involution induced by $j$ is sent  to itself, in the dihedral case). 

There are other cases in which a subscript is necessary. For instance, for Family 33 the relevant isomorphism $f$ between $D^*_{8m}/C_{2m}$ and $D^*_{8n}/C_{2n}$, which are isomorphic to the dihedral group of four elements,
is defined by 
\begin{equation}\label{eq:extra isom}
f[e^{i\pi/m}]=[j]\qquad\text{and}\qquad f[j]=[e^{i\pi/n}]~.
\end{equation} Then the group with data $(D^*_{8m},C_{2m},D^*_{8n},C_{2n},f)$ is not conjugate to Family 11 with $r=2$ (in which case the isomorphism between $D^*_{8m}/C_{2m}$ and $D^*_{8n}/C_{2n}$ is trivial) unless $m=1$ or $n=1$. Indeed in this latter case one has $(L,L_K)=(D_8^ *,C_2)$ (or the same for $(R,R_K)$) and the equivalence classes $[j]=\{\pm j\}$ and $[i]=\{\pm i\}$ are conjugate in $S^ 3$ (for instance by $(i+j)/\sqrt{2}$). This  explains Family 33 in Table \ref{subgroup}, which is one of the families missing in the original list of \cite{duval}, with the restriction $m\neq 1$ and $n\neq 1$

A very similar behaviour appears for Family $33'$, which is another additional missing family in Du Val's original list, together with Family 34. For more details we refer to \cite[Section 2.1]{Mecchia-seppi2}.

It is finally important to remark that several families, which we will call for instance Families 2bis, 3bis and so on, are to be added to Table \ref{subgroup} for the following reason.  When the pairs $(L,L_K)$ and $(R,R_K)$ do not coincide, the groups determined by the 5-tuples $(L,L_K,R,R_K,\phi)$ and $(R,R_K,L,L_K,\phi^{-1})$ are conjugate in $\O(4)$ by means of the  orientation-reversing isometry of $S^3$, sending each quaternion  to its inverse, but \emph{not} in $\SO(4)$. For example, we shall call Family 2bis the family of groups with data $(D^*_{4m}/D^*_{4m},C_{2n}/C_{2n})$; on the other hand, there is no Family 1bis since for each group in Family 1, its conjugate by an orientation-reversing isometry is already contained in Family 1 itself; as a final example, there is no Family 20bis as switching the roles of left and right multiplication gives rise to the same group. This analysis is developed more carefully in \cite[Section 3.2]{Mecchia-seppi2}, leading also to the computation of orientation-reversing self-isometries.





\begin{table}
\centering
\begin{adjustbox}{width=0.71\textwidth}
\begin{tabular}{|l|c|c|c|}
\hline
 & $\widehat G$ & order of $G$ & \\
\hline
 1. & $(C_{2mr}/C_{2m},C_{2nr}/C_{2n})_s$ & $2mnr$ & $\gcd(s,r)=1$ \\  
 $1^{\prime}$. & $(C_{mr}/C_{m},C_{nr}/C_{n})_s$ & $(mnr)/2$ & $\gcd(s,r)=1$ $\gcd(2,n)=1$ \\
&&& $\gcd(2,m)=1$  $\gcd(2,r)=2$\\
 2. & $(C_{2m}/C_{2m},D^*_{4n}/D^*_{4n})$ & $4mn$ &  \\ 
 3. & $(C_{4m}/C_{2m},D^*_{4n}/C_{2n})$ & $4mn$ &   \\ 
 4. & $(C_{4m}/C_{2m},D^*_{8n}/D^*_{4n})$ & $8mn$ &   \\ 
 5. & $(C_{2m}/C_{2m},T^*/T^*)$ & $24m$ &   \\
 6. & $(C_{6m}/C_{2m},T^*/D^*_{8})$ & $24m$ &   \\ 
 7. & $(C_{2m}/C_{2m},O^*/O^*)$ & $48m$ &   \\
 8. & $(C_{4m}/C_{2m},O^*/T^*)$ & $48m$ &   \\ 
 9. & $(C_{2m}/C_{2m},I^*/I^*)$ & $120m$ &   \\ 
 10. & $(D^*_{4m}/D^*_{4m},D^*_{4n}/D^*_{4n})$ & $8mn$ &   \\
 11. & $(D^*_{4mr}/C_{2m},D^*_{4nr}/C_{2n})_s$ & $4mnr$ & $\gcd(s,r)=1$ \\ 
 $11^{\prime}$. & $(D^*_{2mr}/C_{m},D^*_{2nr}/C_{n})_s$ & $mnr$ &  $\gcd(s,r)=1$ $\gcd(2,n)=1$  \\
&&&  $\gcd(2,m)=1$ $\gcd(2,r)=2$\\
 12. &  $(D^*_{8m}/D^*_{4m},D^*_{8n}/D^*_{4n})$ & $16mn$ & \\
 13. &  $(D^*_{8m}/D^*_{4m},D^*_{4n}/C_{2n})$ & $8mn$ & \\
 14. &  $(D^*_{4m}/D^*_{4m},T^*/T^*)$ & $48m$ & \\
 15. &  $(D^*_{4m}/D^*_{4m},O^*/O^*)$ & $96m$ & \\
16. &  $(D^*_{4m}/C_{2m},O^*/T^*)$ & $48m$ & \\
17. &  $(D^*_{8m}/D^*_{4m},O^*/T^*)$ & $96m$ & \\
18. & $(D^*_{12m}/C_{2m},O^*/D^*_{8})$ & $48m$ & \\
19. & $(D^*_{4m}/D^*_{4m},I^*/I^*)$ & $240m$ & \\
20. & $(T^*/T^*,T^*/T^*)$ & $288$ & \\
21. & $(T^*/C_2,T^*/C_2)$ & $24$ & \\
$21^{\prime}.$ & $(T^*/C_1,T^*/C_1)$ & $12$ & \\
22. & $(T^*/D^*_{8},T^*/D^*_{8})$ & $96$ & \\
23. & $(T^*/T^*,O^*/O^*)$ & $576$ & \\
24. & $(T^*/T^*,I^*/I^*)$ & $1440$ & \\
25. & $(O^*/O^*,O^*/O^*)$ & $1152$ & \\
26. & $(O^*/C_2,O^*/C_2)$ & $48$ & \\
$26^{\prime}.$ & $(O^*/C_1,O^*/C_1)_{Id}$ & $24$ & \\
$26^{\prime\prime}.$ & $(O^*/C_1,O^*/C_1)_f$ & $24$ & \\
27. & $(O^*/D^*_{8},O^*/D^*_{8})$ & $192$ & \\
28. & $(O^*/T^*,O^*/T^*)$ & $576$ & \\
29. & $(O^*/O^*,I^*/I^*)$ & $2880$ & \\
30. &   $(I^*/I^*,I^*/I^*)$ & $7200$ & \\
31. &   $(I^*/C_2,I^*/C_2)_{Id}$ & $120$ & \\
$31^{\prime}.$ & $(I^*/C_1,I^*/C_1)_{Id}$ & $60$ & \\
32. &   $(I^*/C_2,I^*/C_2)_{f}$ & $120$ & \\
$32^{\prime}.$ & $(I^*/C_1,I^*/C_1)_f$ & $60$ & \\
33. & $(D^*_{8m}/C_{2m},D^*_{8n}/C_{2n})_f$ & $8mn$ &   $m\neq 1$  $n\neq 1$. \\ 
 $33^{\prime}$. & $(D^*_{8m}/C_{m},D^*_{8n}/C_{n})_f$ & $4mn$ &  $\gcd(2,n)=1 \gcd(2,m)=1$ \\
&&& $m\neq 1$ and $n\neq 1$.   \\
34. &  $(C_{4m}/C_{m},D^*_{4n}/C_{n})$ & $2mn$ & $\gcd(2,n)=1 \gcd(2,m)=1$ \\
\hline
\end{tabular}
\end{adjustbox}
\bigskip
\caption{Finite subgroups of $\SO(4)$}
\label{subgroup}
\end{table}


\section{Invariant Seifert fibrations} \label{sec: Seifert S3 invariant}

We will now compare the above list of conjugacy classes of finite subgroups of $\SO(4)$ with the classification of subgroups which preserve Seifert fibrations of $S^3$, up to  fibration-preserving conjugacy. 

\subsection{Seifert fibrations of $S^3$ revisited} We have already introduced the Seifert fibrations of $S^3$ in Section \ref{sec seifert s3}; let us now give a more geometric description. 


The \emph{Hopf fibration} can be defined by means of the action of $S^1$ on $S^3$ by left multiplication: $(e^{i\theta},q)\mapsto e^{i\theta}q$ for $q\in S^3$. The fibers of the Hopf fibration are then the orbits of this (free) action and the projection map can be written as $\pi:S^3\rightarrow S^2$ expressed by 
$$\pi(z_1+z_2 j)=\frac{z_1}{z_2}$$
where we are identifying the target $S^2$ with the Riemann sphere in the model $\mathbb{C}\cup\left\{\infty\right\}$. The subgroup $\NN$ of $\SO(4)$ which \emph{preserves} the Hopf fibration, meaning that $\gamma$ induces a diffeomorphism $\rho(\gamma)$ of the base $S^2$, coincides with 
$\Norm_{S^3\times S^3}(S^1\times\{1\})=\Norm_{S^3}(S^1)\times S^3$,
where it is easily checked that an element $w_1+w_2j$ normalizes $S^1$ if and only if $w_1=0$ or $w_2=0$. Hence $\Norm_{S^3}(S^1)=\{w_1+w_2 j\,|\,w_1=0\text{ or }w_2=0\}=\O(2)^*$ and
$$\NN=\Norm_{S^3\times S^3}(S^1\times\{1\})=\O(2)^*\times S^3~.$$
It is also useful for the following to compute the induced action of $\NN$ on the base $S^2$ of the fibration. For this purpose, it suffices to observe that elements of the form $(e^{i\theta},0)\in\NN$ clearly act trivially on the base; that $(j,0)$ (and therefore all elements of the form $(e^{i\theta}j,0)$) acts by the antipodal map of $S^2$; on the other hand 
$(0,w_1+w_2 j)$  induces the action \begin{equation} \label{equation induced action}
\lambda\mapsto \frac{\overline w_1 \lambda+\overline w_2}{-w_2\lambda+w_1}\,
\end{equation}
on the base $\mathbb{C}\cup\left\{\infty\right\}$.
For instance, elements of the form $(0,e^{i\theta})$ act on $S^2$ by rotations of angle $2\theta$ fixing the poles  $0$ and $\infty$ while $(0,j)$ is a rotation of order two fixing $i$ and $-i$ and switching $0$ and $\infty$. The \emph{anti-Hopf fibration} is then obtained by composing $\pi$ with an orientation-reversing isometry of $S^3$. Choosing the isometry $q\to q^{-1}$, we obtain the expression $(z_1+z_2 j)\mapsto -{\overline z_1}/{z_2}$.

In general the Seifert fibrations of $S^3$ introduced in Section \ref{sec seifert s3}, can be obtained by the action of $S^1$ defined by $(e^{i\theta},z_1+z_2j)\mapsto (e^{iv\theta}z_1,e^{iu\theta}z_2)$ or $(e^{i\theta},z_1+z_2j)\mapsto (e^{-iv\theta}z_1,e^{iu\theta}z_2)$, for $q\in S^3$ and $u,v$ coprime integers.  We call these fibrations \emph{standard} and any Seifert fibration of $S^3$ can be mapped by an orientation-preserving diffeomorphism to a standard one.

The projection of a standard fibration can be written as  $$\pi(z_1+z_2 j)=\frac{z_1^u}{z_2^v}\qquad \text{ or }\qquad\pi(z_1+z_2 j)=\frac{\overline z_1^u}{z_2^v}~.$$ In the first case, the normalizer of the action consists of elements $(w_1+w_2 j,u_1+u_2 j)$ provided $w_2=u_2=0$ or $w_1=u_1=0$ (unless $u=v=1$). From the definition of the $S^1$ action, it can be checked directly that the base orbifold is $S^2(u,v)$ where the two cone points are the images of the fibers $z_1=0$ and $z_2=0$, with local invariants $\bar v/u$ and $\bar u/v$ where $u\bar u +v \bar v=1$.

If $u=v=1$ we recover the Hopf fibration. (Indeed when $u=1$ or $v=1$ the point is regular and the corresponding fiber generic). 

Similarly as before, it is not necessary to repeat the analysis for the fibrations of the second type, as for every pair $(u,v)$ one can pre-compose $\pi$ with the orientation-reversing isometry $q\mapsto q^{-1}$ to obtain a new fibration of $S^3$ which is inequivalent to the previous ones in the category of oriented Seifert fibered manifolds, but equivalent in the category of (unoriented) Seifert fibered manifolds.
Observe moreover that these fibrations always have bad base orbifolds,  with the only exceptions of the Hopf and anti-Hopf fibrations.

\subsection{Seifert fibrations from Du Val's list}

In this subsection we analyze which standard fibrations of $S^3$ are left invariant by the subgroups in the Du Val's list.
Going back to Du Val's list of subgroups of $\SO(4)$ (Table \ref{subgroup}), the groups which preserve the Hopf fibration are those with $L=C_{m}$ or $L=D_{2m}^*$, for some $m$. That is, by Families 1 to 19, 33, $33'$, 34, and moreover Families 2 bis, 3 bis, 4 bis, 13 bis and 34 bis. In \cite{mecchia-seppi},  the invariants of the Seifert fibration induced in the quotient by each of these groups was computed, and we report the results in Table \ref{topolino}. We omitted the results for Families 1, $1'$, 11 and $11'$ which have a more complicated expression (see \cite[Tables 2 and 3]{mecchia-seppi}).

The other fibrations of type $\pi(z_1+z_2 j)={{z}_1^u}/{z_2^v}$ are left invariant only by groups in Family 1,$1'$,11,$11'$ and the spherical orbifolds obtained as quotients by these groups have an infinite number of inequivalent fibrations.
To get a similar analysis for the remaining fibrations it suffices to note that the orientation-reversing isometry $q\mapsto q^{-1}$ maps the fibration $\pi(z_1+z_2 j)={\overline{z}_1^u}/{z_2^v}$ to $\pi(z_1+z_2 j)={z_1^u}/{z_2^v}$, and a group preserves a fibration $\pi(z_1+z_2 j)={\overline{z}_1^u}/{z_2^v}$ if and only if its {conjugate} by $q\mapsto q^{-1}$ preserves $\pi(z_1+z_2 j)={z_1^u}/{z_2^v}$. 

\begin{Remark}\label{anti-Hopf}
The Seifert invariants of the quotient orbifold $S^3/G$ induced by the anti-Hopf fibration can be obtained from those arising from the  Hopf fibration. In fact if a group $G$ with data $(L,L_K,R,R_K,\phi)$ preserves the anti-Hopf fibration, then the group $G'$ given by $(R,R_K,L,L_K,\phi^{-1})$ preserves the Hopf fibration: the Seifert  fibration induced by the anti-Hopf fibration on  $S^3/G$ has the same base orbifold of that induced by the Hopf fibration on $S^3/G'$ while the numerical invariants of $S^3/G$ are the opposite of those of  $S^3/G'$ (Remark \ref{rmk:change orientation}).
\end{Remark}

\begin{table}
\centering
\begin{tabular}{|l|c|c|c|c|c|}
\hline 
  & group & e & base orbifold & invariants & case \\ 
\hline
 2. & $(C_{2m}/C_{2m},D^*_{4n}/D^*_{4n})$ & $-\frac{m}{n}$ & $S^2(2,2,n)$ & $\frac{m}{n},\frac{m}{2},\frac{m}{2}$ & \\ 
  2.bis &  $(D^*_{4m}/D^*_{4m},C_{2n}/C_{2n})$ & $-\frac{m}{n}$ & $D^2(n;)$ & $\frac{m}{n}$ & $n$ even \rule{0pt}{3ex}\\
 &  & & $\mathbb{R}P^2(n)$ & $\frac{m}{n}$ & $n$ odd \\
 3. &  $(C_{4m}/C_{2m},D^*_{4n}/C_{2n})$ & $-\frac{m}{n}$ & $S^2(2,2,n)$ & $\frac{m}{n},\frac{m+1}{2},\frac{m+1}{2}$ & \rule{0pt}{3ex}\\ 
 3.bis &  $(D^*_{4m}/C_{2m},C_{4n}/C_{2n})$ & $-\frac{m}{n}$ & $D^2(n;)$ & $\frac{m}{n}$ & $n$ odd \rule{0pt}{3ex}\\
 &  & & $\mathbb{R}P^2(n)$ & $\frac{m}{n}$ & $n$ even \\

 4. &  $(C_{4m}/C_{2m},D^*_{8n}/D^*_{4n})$ & $-\frac{m}{2n}$ & $S^2(2,2,2n)$ & $\frac{m+n}{2n},\frac{m}{2},\frac{m+1}{2}$ & \rule{0pt}{3ex}\\ 
 4.bis & $(D^*_{8m}/D^*_{4m},C_{4n}/C_{2n})$ & $-\frac{m}{2n}$ & $D^2(2n;)$ & $\frac{m+n}{2n}$ & \rule{0pt}{3ex}\\
 5. & $(C_{2m}/C_{2m},T^*/T^*)$ & $-\frac{m}{6}$ & $S^2(2,3,3)$ & $\frac{m}{2},\frac{m}{3},\frac{m}{3}$ &  \rule{0pt}{3ex}\\ 
 6. & $(C_{6m}/C_{2m},T^*/D^*_{8})$ & $-\frac{m}{6}$ & $S^2(2,3,3)$ & $\frac{m}{2},\frac{m+1}{3},\frac{m+2}{3}$ &  \rule{0pt}{3ex}\\ 
7. & $(C_{2m}/C_{2m},O^*/O^*)$ & $-\frac{m}{12}$ & $S^2(2,3,4)$ & $\frac{m}{2},\frac{m}{3},\frac{m}{4}$ &  \rule{0pt}{3ex}\\ 
 8. &  $(C_{4m}/C_{2m},O^*/T^*)$ & $-\frac{m}{12}$ & $S^2(2,3,4)$ & $\frac{m+1}{2},\frac{m}{3},\frac{m+2}{4}$ &  \rule{0pt}{3ex}\\ 
 9. &  $(C_{2m}/C_{2m},I^*/I^*)$ & $-\frac{m}{30}$ & $S^2(2,3,5)$ & $\frac{m}{2},\frac{m}{3},\frac{m}{5}$ &  \rule{0pt}{3ex}\\ 
10. &  $(D^*_{4m}/D^*_{4m},D^*_{4n}/D^*_{4n})$ & $-\frac{m}{2n}$ & $D^2(;2,2,n)$ & $\frac{m}{n},\frac{m}{2},\frac{m}{2}$ & $n$ even  \rule{0pt}{3ex}\\
 & & & $D^2(2;n)$ & $\frac{m}{n},\frac{m}{2}$ & $n$ odd \\
 12. &   $(D^*_{8m}/D^*_{4m},D^*_{8n}/D^*_{4n})$ & $-\frac{m}{4n}$ & $D^2(;2,2,2n)$ & $\frac{m+n}{2n},\frac{m}{2},\frac{m+1}{2}$ & \rule{0pt}{3ex}\\ 
  13. &   $(D^*_{8m}/D^*_{4m},D^*_{4n}/C_{2n})$ & $-\frac{m}{2n}$ & $D^2(;2,2,n)$ & $\frac{m}{n},\frac{m+1}{2},\frac{m+1}{2}$ & $n$ even  \rule{0pt}{3ex}\\
 & & & $D^2(2;n)$ & $\frac{m}{n},\frac{m+1}{2}$ & $n$ odd \\
13.bis &  $(D^*_{4m}/C_{2m},D^*_{8n}/D^*_{4n})$ & $-\frac{m}{2n}$ & $D^2(;2,2,n)$ & $\frac{m}{n},\frac{m}{2},\frac{m}{2}$ & $n$ odd  \rule{0pt}{3ex}\\
 & & & $D^2(2;n)$ & $\frac{m}{n},\frac{m}{2}$ & $n$ even \\

14.  & $(D^*_{4m}/D^*_{4m},T^*/T^*)$ & $-\frac{m}{12}$ & $D^2(3;2)$ & $\frac{m}{2},\frac{m}{3}$ & \rule{0pt}{3ex}\\
 
15. & $(D^*_{4m}/D^*_{4m},O^*/O^*)$ & $-\frac{m}{24}$ & $D^2(;2,3,4)$ & $\frac{m}{2},\frac{m}{3},\frac{m}{4}$ &  \rule{0pt}{3ex}\\ 
16.  & $(D^*_{4m}/C_{2m},O^*/T^*)$ & $-\frac{m}{12}$ & $D^2(;2,3,3)$ & $\frac{m}{2},\frac{m}{3},\frac{m}{3}$ &  \rule{0pt}{3ex}\\
17.  &  $(D^*_{8m}/D^*_{4m},O^*/T^*)$ & $-\frac{m}{24}$ & $D^2(;2,3,4)$ & $\frac{m+1}{2},\frac{m}{3},\frac{m+2}{4}$ &  \rule{0pt}{3ex}\\ 
18. & $(D^*_{12m}/C_{2m},O^*/D^*_{8})$ &  $-\frac{m}{12}$ & $D^2(;2,3,3)$ & $\frac{m}{2},\frac{m+1}{3},\frac{m+2}{3}$ &  \rule{0pt}{3ex}\\ 

19. &  $(D^*_{4m}/D^*_{4m},I^*/I^*)$ & $-\frac{m}{60}$ & $D^2(;2,3,5)$ & $\frac{m}{2},\frac{m}{3},\frac{m}{5}$ &  \rule{0pt}{3ex}\\ 
33. &  $(D^*_{8m}/C_{2m},D^*_{8n}/C_{2n})_f$ & $-\frac{m}{2n}$ & $D^2(;2,2,n)$ & $\frac{m}{n},\frac{m+1}{2},\frac{m+1}{2}$ & $n$ odd  \rule{0pt}{3ex}\\
&  & & $D^2(2;n)$ & $\frac{m}{n},\frac{m+1}{2}$ & $n$ even \\
$33^{\prime}$. &   $(D^*_{8m}/C_{m},D^*_{8n}/C_{n})_f$ & $-\frac{m}{4n}$ & $D^2(;2,2,n)$ & $\frac{[(m+n)/2]}{n},\frac{m}{2},\frac{m+1}{2}$ & $m,n$ odd \rule{0pt}{3ex}\\

34. &   $(C_{4m}/C_{m},D^*_{4n}/C_{n})$ & $-\frac{m}{2n}$ & $S^2(2,2,n)$ & $\frac{[(m+n)/2]}{n},\frac{m}{2},\frac{m+1}{2}$ & $m,n$ odd \rule{0pt}{3ex}\\
34.bis &  $(D^*_{4m}/C_{m},C_{4n}/C_{n})$ & $-\frac{m}{2n}$ & $D^2(n;)$ & $\frac{[(m+n)/2]}{n}$ & $m,n$ odd \rule{0pt}{3ex}\\

\hline 
\end{tabular}
\bigskip
\caption{Computation of local invariants from Du Val's presentation}
\label{topolino}
\end{table}

\subsection{Spherical orbifolds with multiple fibrations}
In the previous sections we discussed which standard Seifert fibrations of $S^3$ are left invariant by the groups in Du Val's list. If a Seifert fibration of $S^3$ is left invariant by a group $G$ acting on $S^3$, the Seifert fibration of $S^3$ induces a Seifert fibration of the quotient orbifold. If two different standard fibrations are left invariant, the induced fibrations of the quotient orbifold are not equivalent.


Moreover, we remark that finite subgroups in Du Val's list can leave  invariant  also  fibrations that are not  standard, which can induce additional fibrations in the quotient orbifolds.  

In this section we explore this phenomenon; the  following lemma proved in \cite[Lemma 5]{Mecchia-seppi2}  shows that it can occur only in some specific cases.  We call a \textit{non-Hopf fibration} a fibration  that cannot be mapped neither to the Hopf nor to the anti-Hopf fibration.

\begin{Lemma} \label{lemma due casi}
Let $G$ be a finite subgroup of $\SO(4)$ leaving invariant a Seifert fibration $\pi$ of $S^3$, then one of the two following conditions is satisfied:
\begin{enumerate}
\item $G$ is conjugate in $\SO(4)$ to a subgroup in Families $1$, $1'$, $11$ or $11'$ and $\pi$ is a non-Hopf fibration;
\item there exists  an orientation-preserving  diffeomorphism $f:S^3\to S^3$ such that   $\pi\circ f$ is the Hopf or the anti-Hopf fibration and $f^{-1}Gf$ is a subgroup of $\SO(4).$
\end{enumerate}
\end{Lemma}

Lemma \ref{lemma due casi} has many interesting consequences.

First, if $L$ and $R$ are both isomorphic to $T^*$, $O^*$ or $I^*$, then no fibration of $S^3$ is preserved by the action of the group $(L,L_K,R,R_K,\phi)$.

Then, if a group $G$ leaves invariant a non-Hopf fibration $\pi$,  the base orbifold of the  fibration induced on $S^3/G$ (which is a bad orbifold) is either a sphere with at most two cone points  or a disk with at most two corner points, since it is obtained as a quotient of the bad 2-orbifold $S^2(u,v)$ with $u$ and $v$ different coprime integers.

Finally, if $G$ leaves  invariant a fibration equivalent to the  Hopf fibration or to the anti-Hopf fibration, we can suppose that the fibration is  standard  and $G$ is a subgroup of $\SO(4).$ We remark that this does not imply that $G$ is a subgroup of the Du Val's list; $G$ is conjugate by an isometry to a group in  the Du Val's  list but this isometry in general does not leave invariant the fibration.

We will focus for a moment on the case of groups of isometries which leave invariant the Hopf fibration.
If $G$ leaves invariant the Hopf fibration, then $G$ is a subgroup of  $\NN=\Norm_{S^3\times S^3}(S^1\times\{1\})$. Moreover,
conjugation of $G$ by elements of  $\NN$ 
respects the Hopf fibration, and therefore induces in the quotient orbifold $S^3/G$ a fibration-preserving isometry. We remark that some of the conjugations used in the work of Du Val list do not have this property. Hence in order to get a  classification of Seifert fibered spherical 3-orbifolds up to orientation-preserving diffeomorphism, we need to classify finite subgroups of $\NN$, up to conjugation in $\NN$. This will result in a new list. 
In the following remark, we explain the differences with respect to Du Val's list (Table \ref{subgroup}).

\begin{Remark}\label{itrefenomeni}

There are three classes of phenomena which can occur for the Hopf fibration, marking the difference with Du Val's list.

\begin{enumerate}

\item  The groups  $(L,L_K,R,R_K,\phi)$ and $(R,R_K,L,L_K,\phi^{-1})$  are conjugate by the orientation reversing isometry $q\mapsto q^{-1}$, which maps the Hopf fibration to the anti-Hopf fibration. Hence these groups are not to be considered equivalent for our purposes, although they are equivalent in Du Val's list. As already explained in Section \ref{subsec finite subgroups} and done in Table \ref{topolino}, if the two families obtained by swapping the roles of $L$ and $R$ do not coincide up to orientation-preserving conjugation, they are distinguished by the suffix ``bis'' added to the number used by Du Val's list. Considering the families leaving invariant the Hopf fibration, this phenomenon is relevant for the families 2, 3, 4, 12 and 34. In these cases the ``bis''   families do also preserve the Hopf fibration, but the Seifert invariant of the quotient orbifolds are obviously given by different formulae, see Table \ref{topolino}.


\item The subgroups generated by $i$ and by $j$ are conjugate in $S^3$, but not in $\O(2)^*$. In Table \ref{subgroup}, the subgroup $D_4^*=\{\pm1,\pm j\}<\O(2)^*$ is not considered since it gives the same group as when it is replaced with $C_4=\{\pm1,\pm i\}$ up to conjugation in $\SO(4)$. To classify subgroups in  $\NN$ it is thus necessary to distinguish the two cases for $L$. Observe moreover that  $D_4^*=\{\pm1,\pm j\}<\O(2)^*$ and  $C_4=\{\pm1,\pm i\}$ are not conjugate  in the normalizer of  the group $D_8$; this implies that a  group with $(L,L_K)=(D_8,C_4)$ is not conjugate  to a group with $(L,L_K)=(D_8,D_4)$.


\item When $L=D_8^*$ and $L_K=C_1$ or $C_2$, the groups of Family 33 with $m=1$, namely $(D^*_{8}/C_{2},D^*_{8n}/C_{2n})_f$ (recalling that the isomorphism $f$ is defined in Subsection \ref{subsec finite subgroups}, Equation \eqref{eq:extra isom}) is conjugate in $S^3$ to the case $r=2$, $m=1$ of Family 11, namely $(D^*_{8}/C_{2},D^*_{8n}/C_{2n})$ (where  the automorphism between $L/L_K$ and $R/R_K$ is the identity). But they are not conjugate in $\NN$ unless $n$ also equals $1$. The same occurs for Family $33'$. Although in Du Val's list Families 33 and $33'$ come with the restriction that $m,n\neq 1$, we will thus consider the case $m=1$, $n\neq 1$ as independent.
\end{enumerate}
Of course the considerations of Remark \ref{itrefenomeni} can be repeated analogously for the anti-Hopf fibration, by switching the roles of $(L,L_K)$ and $(R,R_K)$.

\end{Remark}

\section{Classification by diffeomorphism type}\label{classification}

The purpose of this section is to determine a classification of spherical fibered three-orbifolds up to orientation-preserving diffeomorphism. By Theorem \ref{derham thm}, this also turns out to be a classification up to isometry. More concretely, we will provide a recipe to determine when two fibered spherical orientable three-orbifolds are diffeomorphic in terms of the invariants of their fibrations. In fact, recall from the classification theorem that two Seifert fibered orbifolds have the same base orbifold and numerical invariants if and only if there exists an orientation-preserving diffeomorphism which preserves the fibration. 

\subsection{General strategy}

By Proposition \ref{prop determine geometry}, the possible base orbifolds for a Seifert fibration of a closed spherical orbifold necessarily have positive Euler characteristic. Hence  the possible base orbifolds are listed in \eqref{X>0 1}, \eqref{X>0 2} and \eqref{X>0 3}.

Spherical orbifolds may admit infinitely many inequivalent fibrations. This occurs only if the underlying manifold is a lens space. In fact, lens spaces admit infinitely many Seifert fibrations  
in the manifold sense, and the classification of these Seifert fibrations up to diffeomorphism is well-understood (see Subsection \ref{sec:facts lens spaces} below). The situation is however more delicate here since orbifolds with underlying manifold a lens space may admit other fibrations which are substantially different (i.e. the base orbifold may have mirror reflectors and corner reflectors). Also, orbifolds which admit infinitely many fibrations do admit one (infinitely many, in fact) with base orbifold $S^2(b_1,b_2)$ or $D^2(;b_1,b_2)$, namely, either a sphere with at most two cone points or a disc with at most two corner reflectors.

We will proceed as follows. We first consider fibrations with base orbifold $\mathcal B$ not of the form $S^2(b_1,b_2)$ or $D^2(;b_1,b_2)$.
We thus distinguish the various cases according to the base orbifold of the fibration $\pi:\mathcal O\to\mathcal B$ and pointing out in each case which other fibrations are admitted by $\mathcal O$. If $\mathcal O$ admits more than one fibration, it will re-appear in different cases, providing in each one the instructions to obtain the other fibrations of the orbifolds. When $\mathcal O$ also admits a fibration with base $S^2(b_1,b_2)$ or $D^2(;b_1,b_2)$, we will point out one of them (there are in fact infinitely many).
In this analysis we use the results contained in Table \ref{topolino} (see also Table 4 of \cite{mecchia-seppi}), which compute the invariants of quotients $S^3/G$ when $G<\SO(4)$ preserves the Hopf fibration (recall also the discussion of Section \ref{sec: Seifert S3 invariant}).

We then deal (Subsection \ref{subsec:infinitely}) with base orbifolds  of the form $S^2(b_1,b_2)$ or $D^2(;b_1,b_2)$. 
Given such a spherical 3-orbifold $\mathcal O$, with underlying space a lens space, producing a ``list'' of \emph{all} the (infinitely many) Seifert fibrations that $\mathcal O$ admits is rather complicated and beyond the scope of this paper. For the manifold case, an algorithm to list all Seifert fibrations on a given lens space, satisfying a prescribed bound on the ``complexity'' is described in \cite{geigeslange}. Here we will rather describe a procedure which permits (algorithmically) to determine whether two Seifert fibered orbifolds with base orbifold a sphere with at most two cone points are diffeomorphic. Then we do the same for two Seifert orbifolds with base a disc with at most two corner reflectors, and finally we point out the (few) cases in which an orbifold admits fibrations of both types. Sometimes these orbifolds admit 
 additional fibrations with different bases, and these will have already been pointed out in Subsection \ref{subsec:finitely}.


\subsection{Finitely many fibrations} \label{subsec:finitely}

Let us now exhibit the (orientation-preserving) diffeomorphisms between spherical orbifolds with (finitely many) different fibrations, in terms of their invariants.

\begin{case}
The base orbifold is  $S^2(2,2,b).$
\end{case}

The families of groups giving $S^2(2,2,b)$ as base orbifold  are 2, 3, 4 and 34 with the fibration induced by the Hopf fibration  and 2bis, 3bis, 4bis and 34bis with the anti-Hopf fibration. 
By Equation~\eqref{somma eulero invarianti} we obtain that the Euler invariant can be represented by a fraction with denominator $2b$. (If $b$ is even, also by a fraction with denominator $b$.) Let us consider the generic spherical fibered orbifold with  base orbifold $S^2(2,2,b)$:

$$\left(S^2(2,2,b);\,\frac{m_1}{2},\frac{m_2}{2},\frac{m_3}{b};\,\frac{a}{2b}\right)~.$$

We remark that if $a>0$ the fibration is induced by the anti-Hopf fibration while if $a<0$ it is induced by the Hopf fibration. Indeed each of the groups considered in this case preserves both and each quotient orbifold admits at least two inequivalent fibrations.  By using Table \ref{topolino}  and Equation~\eqref{somma eulero invarianti} we can compute the Seifert invariants of both fibrations for each  quotient orbifold (see also Remark~\ref{anti-Hopf}).

\begin{itemize}

\item If $m_1=m_2=0$ then  $a$ is even, then necessarily $m_3 \equiv - a/2 \,(\mod \,b)$. By comparing  the quotient fibrations of Families 2 and 2bis, or 3 and 3bis, we get:
		$$\left(S^2(2,2,b);\,\frac{0}{2},\frac{0}{2},-\frac{a/2}{b};\,\frac{a/2}{b}\right)\cong \left(D^2(|a/2|;);\,\frac{b}{a/2};\,-\frac{b}{a/2};0\right)$$

\item If $m_1=m_2=1$ then  $a$ is again even, then $m_3 \equiv - a/2 \,(\mod \,b)$ and we get analogously:
$$\left(S^2(2,2,b);\,\frac{1}{2},\frac{1}{2},-\frac{a/2}{b};\,\frac{a/2}{b}\right)\cong \left(\mathbb{R}P^2(|a/2|);\,\frac{b}{a/2};\,-\frac{b}{a/2}\right)$$
Note that this family of orbifolds contains  the prism manifolds, one of the families with multiple fibrations in the manifold case, see for instance \cite[Theorem 2.3]{hatchernotes}.

\item 		If $m_1\neq m_2$ we can suppose $m_1=0$ and $m_2=1$; in this case Equation~\eqref{somma eulero invarianti} implies that $a$ and $b$ have the same parity, and $m_3\equiv (a+b)/2\,(\mod \,b)$. From Families 4 and $34$ and their bis versions we get:  

$$\left(S^2(2,2,b);\,\frac{0}{2},\frac{1}{2},-\frac{(a+b)/2}{b};\,\frac{a}{2b}\right)\cong \left(D^2(|a|;);\,\frac{(a+b)/2}{a};\,-\frac{b}{2a};1\right)$$
\end{itemize}
We remark that, here and in what follows, when the base orbifold has mirror reflectors, the invariant $\xi$ is determined by the other invariants (see \cite[Corollary 2.9]{dunbar2}).

Some of these orbifolds admit further inequivalent fibrations because  $i$ and $j$ are not conjugated in $\O(2)^*$ (see Phenomenon 2 described in Remark~\ref{itrefenomeni}).  For example, in Family 2 the groups  with $m=2$ are conjugate in $\SO(4)$ to groups with $m=1$  in Family 10 but this conjugation cannot be performed in   $\NN=\Norm_{S^3\times S^3}(S^1\times\{1\}).$ This implies that an isometric copy of the Hopf fibration in non-standard position is left invariant by  $(C_{4}/C_{4},D^*_{4n}/D^*_{4n})$ inducing on the quotient orbifold a different Seifert fibration. A very similar situation  holds also for the Family 2 bis and the anti Hopf-fibration. 
If $b$ is even we finally obtain that:
$$\left(S^2(2,2,b);\,\frac{0}{2},\frac{0}{2},\pm\frac{2}{b};\,\mp\frac{2}{b}\right)\cong \left(D^2(;2,2,b);;\,\frac{1}{2},\frac{1}{2},\pm\frac{1}{b};\,\mp\frac{1}{2b};1\right)~,$$
and  if $b$ is odd we get:

$$\left(S^2(2,2,b);\,\frac{0}{2},\frac{0}{2},\pm\frac{2}{b};\,\mp\frac{2}{b}\right)\cong \left(D^2(2;b);\,\frac{1}{2};\,\pm\frac{1}{b};\,\mp\frac{1}{2b};1\right)~.$$
Hence we can conclude that $\left(S^2(2,2,b);\,0/2,0/2,\pm 2/b;\,\mp 2/b \right)$ admits three different fibrations for every $b$. 

The analogous situation occurs for groups with $m=1$ in Family 4 and groups of Family 13 bis with $m=1$. So we consider $b$ even and we obtain that   base of the extra fibration depends on the parity of $b/2$; in fact if $b/2$ is odd we obtain:

$$\left(S^2(2,2,b);\,\frac{0}{2},\frac{1}{2},\pm\frac{1+b/2}{b};\,\mp\frac{1}{b}\right)\cong \left(D^2(;2,2,b/2);;\,\frac{1}{2},\frac{1}{2},\pm\frac{1}{b/2};\,\mp\frac{1}{b};1\right)$$
while if $b/2$ is even we get:

 $$\left(S^2(2,2,b);\,\frac{0}{2},\frac{1}{2},\pm\frac{1+b/2}{b};\,\mp\frac{1}{b}\right)\cong \left(D^2(2;b/2);\,\frac{1}{2};\,\pm\frac{1}{b/2};\,\mp\frac{1}{b/2};1\right).$$
Also these orbifolds admit three fibrations.

Finally the groups in Families 3 and 34 with $m=1$ are conjugate in $\SO(4)$ (but not in   $\NN$) to groups in Family 11 and $11'$; the same holds for the groups in Families 3bis and 34bis with $n=1$ and the anti-Hopf fibration. We can conclude   that the orbifolds $$\left(S^2(2,2,b);\,\frac{0}{2},\frac{0}{2},\pm\frac{1}{b};\,\mp\frac{1}{b}\right)\qquad\text{and}\qquad\left(S^2(2,2,b);\,\frac{0}{2},\frac{1}{2},\pm\frac{(1+b)/2}{b};\,\mp\frac{1}{2b}\right)\text{with }b\text{ odd}$$ admit infinitely many fibrations with base orbifold a disk with at most two corner points. These cases are treated in Subsection \ref{subsec:infinitely} since they need a different approach; here we only exhibit a single fibration with such base (whose Seifert invariants are computed by using \cite[Tables 2 and 3]{mecchia-seppi}):
$$\left(S^2(2,2,b);\,\frac{0}{2},\frac{0}{2},\pm\frac{1}{b};\,\mp\frac{1}{b}\right)\cong\left(D^2(;b,b);\,\pm\frac{1}{b},\pm\frac{1}{b};\mp\frac{1}{b};0\right)$$
for the former, and 
$$\left(S^2(2,2,b);\,\frac{0}{2},\frac{1}{2},\pm\frac{(1+b)/2}{b};\,\mp\frac{1}{2b}\right)\cong\left(D^2(;b,b);\,\pm\frac{(1+b)/2}{b},\pm\frac{(1+b)/2}{b};\mp\frac{1}{2b};1\right) \text{with } b \text{ odd}$$
for the latter.

\begin{case}
The base orbifold is  either $D^2(b;)$ or $\mathbb{R}P^2(b).$
\end{case}

In these cases the families to be considered  are 2, 3, 4 and 34 with the fibration induced by the anti-Hopf fibration  and 2bis, 3bis, 4bis and 34bis with the Hopf fibration. Again each of these groups preserves both the Hopf and the anti-Hopf fibration of $S^3$.  The relations between the two fibrations induced in  the quotient orbifold can be deduced from the previous case, since one of the two fibrations has $S^2(2,2,b)$ as base orbifold. 

It remains to consider the extra fibrations caused by Phenomenon 2 in Remark~\ref{itrefenomeni}. 
First, groups of Family 2 with $m=2$ are conjugated to groups of  Family 10 with $m=1$, but the conjugating  elements  are not contained in $\NN$; this implies that the  groups of Families 2 with $m=2$ leave invariant the standard Hopf fibration, the standard anti-Hopf fibration and a non-standard Hopf fibration whose invariants can be obtained considering Family 10 in   Table \ref{topolino}. An analogous situation occurs for Family 2bis when $n=2$. This implies that the orbifolds  $\left(D^2(2;);\frac{b}{2};\,;\pm\frac{b}{2};0\right)$ admit three inequivalent Seifert fibrations.
Indeed, the three fibrations has been already described in the previous case:
$$\left(D^2(2;);\frac{1}{2};\,;\pm\frac{b}{2};0\right)\cong \left(S^2(2,2,b);\,\frac{0}{2},\frac{0}{2},\pm\frac{2}{b};\,\mp\frac{2}{b}\right)\cong \left(D^2(2;b);\,\frac{1}{2};\,\pm\frac{1}{b};\,\mp\frac{1}{2b};1\right)$$
when $b$ is odd and strictly greater than 1, and 

$$\left(D^2(2;);\frac{0}{2};\,;\pm\frac{b}{2};0\right)\cong \left(S^2(2,2,b);\,\frac{0}{2},\frac{0}{2},\pm\frac{2}{b};\,\mp\frac{2}{b}\right)\cong \left(D^2(;2,2,b);;\,\frac{1}{2},\frac{1}{2},\pm\frac{1}{b};\,\mp\frac{1}{2b};1\right)$$
when $b$ is even.

{The groups in Families 2, 3 and 34 with $n=1$  and the ones in Families 2bis, 3bis and 34bis with $m=1$ are conjugate to groups in Families 1 and $1'$. This implies that the following orbifolds admit infinitely many fibrations with base orbifold $S^2$ with at most two cone points, see  Subsection \ref{subsec:infinitely};  here we only exhibit a single fibration with such base: }

\begin{align*}
\left(D^2(b;);\pm \frac{1}{b};\,;\mp\frac{1}{b};0\right)& \cong \left(S^2(b,b);\,\pm\frac{2}{b},\pm\frac{2}{b};\,\mp\frac{4}{b}\right) \text{with } b \text{ even}\\ 
\left(D^2(b;);\pm \frac{1}{b};\,;\mp\frac{1}{b};0\right) &\cong \left(S^2(2b,2b);\,\pm\frac{1+b}{2b},\pm\frac{1+b}{2b};\,\mp\frac{1}{b}\right) \text{with } b \text{ odd}\\
\left(\mathbb{R}P^2(b);\pm \frac{1}{b};\mp\frac{1}{b}\right)&\cong  \left(S^2(b,b);\,\pm\frac{2}{b},\pm\frac{2}{b};\,\mp\frac{4}{b}\right) \text{with } b \text{ odd}\\
\left(\mathbb{R}P^2(b);\pm \frac{1}{b};\mp\frac{1}{b}\right)&\cong \left(S^2(2b,2b);\,\pm\frac{1+b}{2b},\pm\frac{1+b}{2b};\,\mp\frac{1}{b}\right) \text{with } b \text{ even}\\
\left(D^2(b;);\pm \frac{(1+b)/2}{b};\,;\mp\frac{1}{2b};1\right) &\cong  \left(S^2(2b,2b);\,\pm\frac{(1+b)/2}{2b},\pm\frac{(1+3b)/2}{2b};\,\mp\frac{1}{2b}\right) \text{with } b \text{ odd}
\end{align*}

\begin{case}
The base orbifold is $D^2(;2,2,b).$
\end{case}

The families we have to consider are 10, 12, 13, 13 bis, 33, $33'.$ 
By Equation~\eqref{somma eulero invarianti} the Euler invariant can be represented by a fraction with denominator $4b$ and here  the generic  fibered orbifold is:

$$\left(D^2(;2,2,b);\,\frac{m_1}{2},\frac{m_2}{2},\frac{m_3}{b};\,\frac{a}{4b};\,\xi\right)~.$$

Each of these groups preserves both the Hopf  and the anti-Hopf fibration, and by using Table \ref{topolino} and  Remark~\ref{anti-Hopf} we compute the Seifert invariants induced by both. 
\begin{itemize}

\item If $m_1=m_2=0$ then  $a$ is even, and $m_3 \equiv a/2 \,(\mod \,b)$. This kind of fibration in the quotient orbifold is induced by the Hopf fibration left invariant by groups in  Family 10 with $m$ and $n$ even, in Family 13 with $m$ odd and $n$ even, in Family 13bis with $m$ even and $n$ odd and in Family 33 with $m$ and $n$ odd. Considering the anti-Hopf fibration this case occurs for groups 
in  Families 10 with $m$ and $n$ even, in Family 13 with $n$ odd and $m$ even, in Family 13bis with $n$ even and $m$ odd and in Family 33 with $m$ and $n$ odd. 
If the fibration considered is induced by the Hopf fibration we compute the invariants of the other fibration induced by the anti-Hopf fibration and viceversa.
Finally we get the following diffeomorphisms:

	$$\left(D^2(;2,2,b);\,;\,\frac{0}{2},\frac{0}{2},-\frac{a/2}{b};\,\frac{a/2}{2b};\,0\right)\cong \left(D^2(;2,2,|a/2|);\,;\,\frac{0}{2},\frac{0}{2},\frac{b}{a/2};\,-\frac{b}{a};\,0\right)$$

\item If $m_1=m_2=1$ then  $a$ is again even, and $m_3 \equiv a/2 \,(\mod \,b)$. Carrying out an analysis similar to  the previous case which involves again Families 10, 13, 13bis  and 33,  we obtain the following diffeomorphisms:
$$\left(D^2(;2,2,b);\,;\,\frac{1}{2},\frac{1}{2},-\frac{a/2}{b};\,\frac{a/2}{2b};\,1\right)\cong \left(D^2(2;|a/2|);\,\frac{0}{2};\,\frac{b}{a/2};\,-\frac{b}{a};\,0\right)$$

\item 		If $m_1\neq m_2$ we can suppose $m_1=0$ and $m_2=1$; in this case Equation~\eqref{somma eulero invarianti} implies that $a$ and $b$ have the same parity, and $m_3 \equiv (a+b)/2\,(\mod \,b)$. Here the we have to consider Families 12 and $33'.$ Finally we get: 

$$\left(D^2(;2,2,b);\,;\,\frac{0}{2},\frac{1}{2},-\frac{(a+b)/2}{b};\,\frac{a}{4b};1\right)\cong \left(D^2(;2,2,|a|);\,;\,\frac{0}{2},\frac{1}{2},\frac{(a+b)/2}{a};\,-\frac{b}{4a};1\right)$$

\end{itemize}

Let us now consider the extra fibrations given by Phenomena 2 and 3 of Remark~\ref{itrefenomeni}. 

Phenomenon 2 involves the groups in  Families 10 and 13bis with $m=1,$ which leave invariant a  non-standard Hopf fibration, and in  Families 10 and 13  with $n=1,$ which leave invariant a non-standard anti-Hopf fibration.
We obtain the following diffeomorphisms: 

$$ \left(D^2(;2,2,b);;\,\frac{1}{2},\frac{1}{2},\pm\frac{1}{b};\,\mp\frac{1}{2b};1\right)\cong \left(D^2(2;);\frac{0}{2};\,;\pm\frac{b}{2};0\right)\cong \left(S^2(2,2,b);\,\frac{0}{2},\frac{0}{2},\pm\frac{2}{b};\,\mp\frac{2}{b}\right)$$
when $b$ is even

$$\left(D^2(;2,2,b);;\,\frac{1}{2},\frac{1}{2},\pm\frac{1}{b};\,\mp\frac{1}{2b};1\right) \cong \left(D^2(2;);\,\frac{0}{2};\,\pm\frac{b}{2};1\right) \cong \left(S^2(2,2,2b);\,\frac{0}{2},\frac{1}{2},\pm\frac{1+b}{2b};\,\mp\frac{1}{2b}\right)$$
when $b$ is odd. These orbifolds  admit three fibrations. 

\smallskip

Moreover, by Phenomenon 3 the following orbifolds

$$\left(D^2(;2,2,b);\,;\,\frac{0}{2},\frac{0}{2},\pm\frac{1}{b};\,\mp\frac{1}{2b};0\right)\text{with } b \text{ odd } \text{and}\,\left(D^2(;2,2,b);\,;\,\frac{0}{2},\frac{1}{2},\pm\frac{(b+1)/2}{b};\,\mp\frac{1}{4b};1\right) \text{with } b \text{ odd}$$
{given by groups in Family 33 with $m=1$ and Family $33'$ with $m=1$   have infinitely many fibrations with base orbifold a disk with two corner points (these groups are conjugate in $\SO(4)$ to groups in Families 11 and $11'$). As usual we describe here one of these fibrations and the general analysis can be found in Subsection \ref{subsec:infinitely}.}

\begin{align*}
\left(D^2(;2,2,b);\,;\,\frac{0}{2},\frac{0}{2},\pm\frac{1}{b};\,\mp\frac{1}{2b};0\right)& \cong \left(D^2(;2b,2b);\,;\,\pm\frac{b+1}{2b},\pm\frac{b+1}{2b};\,\mp\frac{1}{2b};1\right) \\ 
\left(D^2(;2,2,b);\,;\,\frac{0}{2},\frac{1}{2},\pm\frac{(b+1)/2}{b};\,\mp\frac{1}{4b};1\right)& \cong \left(D^2(;2b,2b);\,;\,\pm\frac{(3b+1)/2}{2b},\pm\frac{(b+1)/2}{2b};\,\mp\frac{1}{4b};1\right)
\end{align*}
where $b$ is an odd integer in both cases.

\begin{case}
The base orbifold is $D^2(2;b).$
\end{case}

The groups giving fibered quotients with base orbifold  $D^2(2;b)$ are contained in Families 10, 13, 13bis, 33. Each of these groups leaves invariant both the Hopf and the anti-Hopf fibration, so the quotient orbifold has at least two fibrations and the possible base orbifolds  are  $D^2(2;b)$ and $D^2(;2,2,b).$ The cases when  $D^2(;2,2,b)$ appears as base orbifold of one of the two fibrations  have been already treated in the previous case. When both fibrations have  $D^2(2;b)$ we obtain the following diffeomorphism:

$$\left(D^2(2;b);\,\frac{1}{2};\,-\frac{a}{b};\,\frac{a}{2b};\,1\right)\cong \left(D^2(2;|a|);\,\frac{1}{2};\,\frac{b}{a};\,-\frac{b}{2a};\,1\right).$$
Phenomena 2 and 3 of Remark~\ref{itrefenomeni}  give extra fibrations to the following orbifolds:

\begin{itemize}
\item  $\left(D^2(2;b);\,\frac{1}{2};\,\pm \frac{1}{b};\,\mp\frac{1}{2b};\,1\right)$  whose fibration is induced  by the Hopf fibration left invariant by  groups in Families 10 and  13bis with $m=1$, and by the anti-Hopf fibration left invariant by groups in Families 10 and  13bis with $n=1$; these orbifolds  have three fibrations and the extra fibrations (already described in the previous case) are: 

$$\left(D^2(2;);\frac{1}{2};\,;\pm\frac{b}{2};0\right)\cong \left(S^2(2,2,b);\,\frac{0}{2},\frac{0}{2},\pm\frac{2}{b};\,\mp\frac{2}{b}\right)$$
when $b$ is odd and 

$$\left(D^2(2;);\,\frac{1}{2};\,\pm\frac{b}{2};1\right) \cong \left(S^2(2,2,2b);\,\frac{0}{2},\frac{1}{2},\pm\frac{1+b}{2b};\,\mp\frac{1}{2b}\right)  $$
when $b$ is even.

\item $\left(D^2(2;b);\,\frac{0}{2};\,\pm \frac{1}{b};\,\mp\frac{1}{2b};\,1\right)$ with $b$ even, whose fibration is induced by the Hopf fibration left invariants by groups in Family 33 with $m=1$; these orbifolds   have infinitely many fibrations with base orbifold a disk with two corner points because the groups are conjugate to groups in Family 11. Here we describe only one of these fibrations:

 $$\left(D^2(2;b);\,\frac{0}{2};\,\pm \frac{1}{b};\,\mp\frac{1}{2b};\,1\right) \cong\left(D^2(;2b,2b);\,;\,\pm\frac{1+b}{2b};\,\pm \frac{1+b}{2b};\,\mp\frac{1}{2b};\,1\right) $$

\end{itemize}
\begin{case}
The base orbifold is $S^2(2,3,b)$ or $D^2(;2,3,b)$ with $b=3,4,5$.
\end{case}

Fibrations having these base orbifold are induced by the Hopf fibration left invariant by groups in Families 5, 6, 7, 8, 9, 14, 15, 16, 17, 18, 19 and by the anti-Hopf fibration left invariant by the bis-versions of the same families. In this case each group leaves invariant exactly one among the Hopf and the anti-Hopf fibration. With few exceptions due to Phenomenon 2 of Remark \ref{itrefenomeni}, these orbifolds have exactly one fibration.

The  sporadic  diffeomorphisms are given by the  correspondences  between Family 5 ($m=2$) and 14 ($m=1$), Family 7 ($m=2$) and 15 ($m=1$), Family 8 and 16 ($m=1$), Family 9 ($m=2$) and 19 ($m=1$), and the analogous correspondence between the bis-versions of the groups. All the quotient orbifolds thus have 2 inequivalent fibrations. We collect here these diffeomorphisms:
\begin{align*}\left(S^2(2,3,3)\,;\frac{0}{2}\,,\pm \frac{2}{3}\,,\pm \frac{2}{3}\,;\mp \frac{1}{3}\right)&\cong \left(D^2(3;2)\,;\pm \frac{1}{3}\,;\pm \frac{1}{2};\mp \frac{1}{12};1\right) \\
\left(S^2(2,3,4)\,;\frac{0}{2}\,,\pm \frac{2}{3}\,,\pm \frac{2}{4}\,;\mp \frac{1}{6}\right)&\cong \left(D^2(;2,3,4)\,;;\frac{1}{2}\,,\pm \frac{1}{3}\,,\pm \frac{1}{4};\mp \frac{1}{24};1\right) \\
\left(S^2(2,3,4)\,;\frac{0}{2}\,,\pm \frac{1}{3}\,,\pm \frac{3}{4}\,;\mp \frac{1}{12}\right)&\cong \left(D^2(;2,3,3)\,;;\frac{1}{2}\,,\pm \frac{1}{3}\,,\pm \frac{1}{3};\mp\frac{1}{12};1\right)  \\
\left(S^2(2,3,5)\,;\frac{0}{2}\,,\pm \frac{2}{3}\,,\pm \frac{2}{5}\,;\mp\frac{1}{15}\right)&\cong \left(D^2(;2,3,5)\,;;\frac{1}{2}\,,\pm \frac{1}{3}\,,\pm \frac{1}{5};\mp\frac{1}{60};1\right). 
\end{align*}


\subsection{Some facts on lens spaces} \label{sec:facts lens spaces}

Let us briefly recall some generalities on lens spaces. 

\begin{Definition}
The lens space $L(p,q)$ is defined as the manifold obtained by gluing two solid tori $T_1$ and $T_2$ by means of an orientation-reversing diffeomorphism of their boundaries which maps a meridian $\mu_1$ of $T_1$ to $p\lambda_2-q\mu_2$, where $\mu_2$ and $\lambda_2$ are a meridian and a longitude of $T_2$.
\end{Definition} Observe that there is no natural choice of longitude $\lambda_i$ on $\partial T_i$, and in fact if $q\equiv q'\text{ mod }p$, then performing a Dehn twist on $T_2$ gives an orientation-preserving diffeomorphism $L(p,q)\cong L(p,q')$.  For the very same reason, the diffeomorphism type of $L(p,q)$ does not depend on the image of a longitude $\lambda_1$ by the diffeomorphism $\partial T_1\to\partial T_2$. 

Suppose $\mathcal M$ is a Seifert fibered manifold  with base surface $S^2$ and at most two cone points, and associated local invariants $\alpha_1/\beta_1$ and $\alpha_2/\beta_2$ for $\alpha_i$ and $\beta_i$ coprimes (we recall that these are the invariants in the classical notation for Seifert manifolds, see Subsection \ref{rmk manifold}), it is not hard to compute the corresponding lens space. See  \cite[Theorem 4.4]{jnlectures} and \cite[Theorem 4.4]{geigeslange} for a more detailed explanation. Let $T_1$ and $T_2$ be the preimages of two discs $D_1$ and $D_2$ on the base $S^2$,  where $D_1$ contains the first cone point and $D_2$ the second, with $\partial D_1=\partial D_2$. Then $T_1$ and $T_2$ are fibered solid tori. With some computations, which are provided in the given references, one finds that $\mathcal M$ is obtained by gluing $T_1$ and $T_2$ in such a way that a meridian $\mu_1$ of $T_1$ is glued to $p\lambda_2-q\mu_2$, and is thus diffeomorphic to the lens space $L(p,q)$ according to our definition above,
for  
\begin{equation} \label{eq:compute lens from invariants}
p=-\det\begin{pmatrix} \alpha_1 & \alpha_2 \\ -\beta_1 & \beta_2 \end{pmatrix}=-\alpha_1\beta_2-\beta_1\alpha_2 \qquad\text{and}\qquad q=-\det\begin{pmatrix} \alpha_1 & \gamma_2 \\ \beta_1 & \delta_2 \end{pmatrix}=\beta_1\gamma_2-\alpha_1\delta_2~,
\end{equation}
where the pair $(\gamma_2,\delta_2)$ satisfies 
$$\det\begin{pmatrix} \alpha_2 & \gamma_2 \\ -\beta_2 & \delta_2 \end{pmatrix}=\alpha_2\delta_2-\beta_2\gamma_2=1~.$$
The choice of such a pair $(\gamma_2,\delta_2)$ is not unique, but any two choices differ by a multiple of $(\alpha_2,\beta_2)$, hence giving the same result up to the usual modulo $p$ ambiguity for $q$. 

\begin{Remark}
We remark that the expressions in \eqref{eq:compute lens from invariants} are slightly different to those recorded in \cite[Theorem 4.4]{jnlectures} and \cite[Theorem 4.4]{geigeslange}: this is because we are adopting a different convention here for the classification data of a Seifert fibration, following \cite{bonahon-siebenmann} and \cite{hatchernotes}. To pass from our convention to that used in \cite{jnlectures} and \cite{geigeslange} one should switch the roles of the $\alpha_i$ and $\beta_i$, and there is also a sign difference.
\end{Remark}

We shall now explain more precisely necessary and sufficient conditions for two lens spaces to be (orientation-preserving) diffeomorphic. A fundamental fact is the following (proved in \cite{MR710104}, see also \cite[Theorem 10.1.12]{martellilibro} and \cite[Theorem 2.5]{hatchernotes}):
\begin{Proposition} \label{prop heegard surface}
Any two Heegard surfaces of genus 1 in a lens space are isotopic. 
\end{Proposition}
Proposition \ref{prop heegard surface} implies that, given any orientation-preserving diffeomorphism between two lens spaces $L(p,q)$ and $L(p',q')$, one can modify the diffeomorphism so that it maps $\partial T_1=\partial T_2$ to $\partial T'_1=\partial T'_2$. The diffeomorphism now either sends $T_i$ to $T_i$, or $T_1$ to $T_2'$ and $T_2'$ to $T_1$. It is classical to show that the first case occurs if and only if 
\begin{equation} \label{eq: lens no switch}
p=p'\quad \text{and}\quad q\equiv q'\text{ mod }p
\end{equation}
whereas the second if and only if 
\begin{equation}\label{eq: lens with switch}
p=p'\quad \text{and}\quad qq'\equiv 1\text{ mod }p~.
\end{equation}
In light of \eqref{eq: lens no switch}, one checks easily that two equivalent presentations of the same Seifert manifold $\mathcal M$ (as explained in Subsection \ref{rmk manifold}) give an equivalent outcome in Equation \eqref{eq:compute lens from invariants}.

By putting together the two cases \eqref{eq: lens no switch} and \eqref{eq: lens with switch} together, one has:
\begin{Proposition} \label{prop lens space classical}
Two lens spaces $L(p,q)$ and $L(p',q')$ have an orientation-preserving diffeomorphism if and only if
\begin{equation} \label{eq:lens general}
p=p'\quad \text{and}\quad q^{\pm 1}\equiv q'\text{ mod }p~.
\end{equation}
\end{Proposition}
\subsection{Infinitely many fibrations} \label{subsec:infinitely}

Let us now get back to the diffeomorphism type of spherical orbifolds. 

\begin{case}
The base orbifold is  a sphere  with at most two cone points.
\end{case}

In this case the underlying topological space of the  spherical 3-orbifold is a lens space and the possible singular set is contained in the union of the cores of the tori giving the standard Heegard decomposition of the lens space. Consequentely the singular set may be empty, a knot or a link with two components. Exactly as in the analogous manifold  case, if  an orbifold $S^3/G$ is contained in this family, then it admits infinitely many fibrations, since (a conjugate of) $G$ leaves invariant  all the non-Hopf fibrations.

Let us consider $\mathcal O$ and $\mathcal O'$ with base orbifolds $S^2(b_1,b_2)$ and $S^2(b_1',b_2')$, namely
\begin{equation}\label{eq:coronavirus}
\left(S^2(b_1,b_2);\frac{a_1}{b_1},\frac{a_2}{b_2};e\right)\qquad \left(S^2(b'_1,b'_2);\frac{a'_1}{b'_1},\frac{a'_2}{b'_2};e'\right)~,
\end{equation}
and explain how one can decide whether $\mathcal O$ and $\mathcal O'$ are diffeomorphic. 

The recipe is as follows. First compute the index of singularity of the preimages of each cone point. These are simply given by $\iota_i=\mathrm{gcd}(a_i,b_i)$ for $\mathcal O$ and $\iota'_i=\mathrm{gcd}(a'_i,b'_i)$, for $i=1,2$, where $a_i/b_i$ and $a_i'/b_i'$ are the local invariants. Then the underlying Seifert fibered manifolds $\mathcal M$ and $\mathcal M'$ are given by the same expressions as in \eqref{eq:coronavirus}, except that one needs to replace $a_i$ by $a_i/\iota_i$, $b_i$ by $b_i/\iota_i$, and similarly for $a_i'$ and $b_i'$.

Recalling the explanation given in Subsection \ref{rmk manifold}, using also the Euler classes $e$ and $e'$, one can easily express the underlying manifolds $\mathcal M$ and $\mathcal M'$ in terms of the classical invariants for Seifert manifolds. These will be of the form:
\begin{equation} \label{eq:underlying}
\mathcal M\cong \left(S^2;{\alpha_1}/{\beta_1},{\alpha_2}/{\beta_2}\right)\qquad \mathcal M'\cong\left(S^2;{\alpha'_1}/{\beta'_1},{\alpha'_2}/{\beta'_2}\right)~,
\end{equation}
respectively, where  $\beta_i=b_i/\iota_i$, $\beta_i'=b_i'/\iota_i'$, the $\alpha_i$ will be integer numbers relatively prime with $\beta_i$ and satisfying $\alpha_i/\beta_i\equiv a_i/b_i\,\mod \,1$, and similarly for the $\alpha_i'$.

There are now various possibilities:
\begin{itemize}
\item If $\{\iota_1,\iota_2\}$ does not equal $\{\iota'_1,\iota'_2\}$, then $\mathcal O$ and $\mathcal O'$ are certainly not diffeomorphic, since a diffeomorphism should respect the singularity indices.
\item If $\iota_1=\iota_2=\iota'_1=\iota'_2$, determining whether $\mathcal O$ and $\mathcal O'$ are diffeomorphic amounts to whether their underlying manifolds $\mathcal M$ and $\mathcal M'$ are diffeomorphic. In fact, by the discussion above, a diffeomorphism between lens spaces always maps a solid torus of the standard Heegard decomposition of $\mathcal M$ to a solid torus of the decomposition of $\mathcal M'$, and one can arrange such a diffeomorphism to map the (singular, if the $\iota_i$ are not $1$) cores of each torus of $\mathcal M$ to cores of each torus of $\mathcal M'$. Therefore, it suffices to compute the underlying lens spaces of $\mathcal O$ and $\mathcal O'$ (computed in the first step) by applying \eqref{eq:compute lens from invariants} to \eqref{eq:underlying}, and determine whether they are diffeomorphic by checking if the classical formula \eqref{eq:lens general} holds (Proposition \ref{prop lens space classical}).
\item The last possibility is when $\{\iota_1,\iota_2\}=\{\iota'_1,\iota'_2\}$ but $\iota_1\neq \iota_2$ (and thus $\iota'_1\neq \iota'_2$). In this case one cannot directly apply the standard classification for the underlying lens spaces, since one has to take into account that the cores of the two solid tori in each Heegard decomposition have different singularity index, and the singularity index must be preserved by orbifold diffeomorphisms. Up to switching the order of cone points, let us assume $\iota_1=\iota_1'$ and $\iota_2=\iota'_2$. In the notation used above in this section, $\mathcal O$ and $\mathcal O'$ are diffeomorphic if and only if there is a diffeomorphism of the underlying lens spaces which maps $T_i$ to $T_i'$, for $i=1,2$. According to \eqref{eq: lens no switch}, this is the case if and only if $p=p'$ and $q,q'$ have the same residue modulo $p$. In conclusion, one has to use again the formulae \eqref{eq:compute lens from invariants} to compute $p,q,p',q'$, and check whether they satisfy \eqref{eq: lens no switch} (instead of \eqref{eq:lens general}).
\end{itemize}

{Finally we remark that there are two spherical  3-orbifolds, each of which admits  both infinitely many  fibrations with base orbifold  a 2-sphere  with two cone points and infinitely many  fibrations  with base orbifold  a disk  with two corner points.} Fibrations having base orbifold a disk  with two corner points will be considered in Case 7 below.

The first orbifold is the quotient by the group $(C_4/C_2,C_4/C_2)$ that can be conjugate ({by an isometry which does not preserve the Hopf fibration}) to the group $(D^*_{4}/C_2, D^*_{4}/C_2).$ The underlying topological space of this 3-orbifold is the 3-sphere and the singular set  is the Hopf link  whith local group of order two. Two possible fibrations are $\left (S^2(2,2);\frac{0}{2},\frac{0}{2};-1\right)$ and $\left(D^2;\,;\,;-1;0\right)$. Note that these  are the two fibrations described in the example of Section \ref{subsec:example nonuniqueness}.

The second orbifold is given by the group $(C_4/C_1,C_4/C_1)$ that can be conjugate to the group $(D^*_{4}/C_1, D^*_{4}/C_1)$. In this case the  underlying topological space is the 3-sphere and the singular set  is the trivial knot  whith local group of order two. Two possible fibrations are $\left(S^2(2,2);\frac{0}{2},\frac{1}{2};-\frac{1}{2}\right)$ and $\left(D^2;\,;\,;-\frac{1}{2};1\right)$.





\begin{case}
The base orbifold is a disk  with at most two corner points.
\end{case}

Also in this case all the orbifolds admit infinitely many fibrations, since all the non-Hopf fibrations are preserved by  the groups involved.

Let $\mathcal{O}$ be a fibered orbifold whose base orbifold is $D^2(n_1,n_2)$.
In this case we can consider the 2-fold branched covering $\mathcal{O}'$ of  $\mathcal{O}$ induced by the 2-fold orbifold covering $S^2(n_1,n_2)\rightarrow  D^2(n_1,n_2)$ obtained by doubling the disk along its boundary. The orbifold $\mathcal{O}'$ has the same local invariants of $\mathcal{O}$. ({It is worth remarking that the preimage of a corner reflector in the base orbifold of $\mathcal{O}$ will be a cone point in the base orbifold of $\mathcal{O}'$, but the associated numerical invariants will be the same.}) The Euler class of $\mathcal{O}'$ is twice as that of  $\mathcal{O}$.

Fibrations  with base orbifold $D^2(n_1,n_2)$ are admitted by  orbifolds given by groups in  Families 11 {and $11'$}. These groups are generalized dihedral groups $\mathbb{Z}_2\ltimes A$ where $A$ is an Abelian normal subgroup of rank at most two. The subgroup $A$ is the unique {index two} Abelian subgroup with this property (with the exception of the cases $A\cong \mathbb{Z}_2$ and $A\cong \mathbb{Z}_2 \times  \mathbb{Z}_2$).

The subgroup $A$ corresponds to the unique 2-fold branched covering  induced by the orbifold covering $S^2(n_1,n_2)\rightarrow  D^2(n_1,n_2)$. Also in the two exceptional cases all the possible 2-fold branched coverings having $S^2(n_1,n_2)$ as base orbifold are diffeomorphic.  Hence we have a 1-1 correspondence between the orbifolds of Case 6 and the orbifolds of Case 7. We  can reduce to the same procedure as in the previous case to decide whether two fibered orbifolds with base orbifolds $D^2(n_1,n_2)$ are diffeomorphic.  

Some of these orbifolds admit also a fibration whose base orbifold is not a disk with at most two corner points. We have already listed these exceptional fibrations in Cases 1-6.

\begin{example}\label{ex:2bridge}

Case 7 includes the orbifolds having  the 3-sphere as underlying  topological space and  a 2-bridge link with local group of order two as singular set.  The 2-bridge links are a well known class of links which are classified by their 2-fold branched coverings (see, e.g., \cite{burde-zieschang}). In fact the 2-fold branched covering of each of these orbifolds along the singular set is a lens space; on the other hand, each  lens space has a unique representation  as 2-fold branched covering  of a link \cite{hodgson-rubinstein}.   Hence, the classification of 2-bridge links (and of our 3-orbifolds) can be deduced  by the classification of lens spaces. We denote by $\mathcal{O}(p,q)$ the 3-orbifold  that gives $L(p,q)$ as 2-fold branched covering. We remark that there exists an orientation-preserving diffeomorphism between $\mathcal{O}(p,q)$ and  $\mathcal{O}(p',q')$ if and only if  $p=p'$ and $q^{\pm 1}\equiv q'\text{ mod }p~.$  

A lens space can be obtained as  the quotient of $S^3$ by a cyclic group acting freely; these groups are contained in Family 1 or $1'$. Each group $(C_{2mr}/C_{2m},C_{2nr}/C_{2n})_s$ in Family 1 is  normalized by the involution  $\Phi(j,j)\in\SO(4)$ which has  non-empty fixed point set;  the group $(C_{2mr}/C_{2m},C_{2nr}/C_{2n})_s$  and the involution $\Phi(j,j)$  generate   $(D^*_{4m}/D^*_{4m},D^*_{4n}/D^*_{4n})_s$ in Family 11. We remark that the underlying topological space of the quotient of  $S^3$ by the group $(D^*_{4m}/D^*_{4m},D^*_{4n}/D^*_{4n})_s$ is always the 3-sphere (see \cite{mecchia-seppi}) and, if $(C_{2mr}/C_{2m},C_{2nr}/C_{2n})_s$ acts freely, the quotient orbifold is of type  $\mathcal{O}(p,q).$ The same holds for groups in Familiy $1'$ which are contained in groups in Family $11'$.  This assures us that each 3-orbifold $\mathcal{O}(p,q)$  is the quotient of $S^3$ by a group contained in Families 11 and $11'$. Therefore $\mathcal{O}(p,q)$ admits infinitely many Seifert fibrations with base orbifolds a disk with at most two corner points.  Given a group in Family 11 and $11'$, by using \cite[Tables 2 and 3]{mecchia-seppi} we can decide if it gives an orbifold of type $\mathcal{O}(p,q)$ and compute $p$ and $q$ 

Some of these orbifolds admit also some fibrations with a different base orbifold. We can analyze this phenomenon considering the diffeomorphisms between fibered spherical 3-orbifolds listed in  Cases 1-6. We obtain the following: $\mathcal{O}(1,0)$, whose singular set is the unknot, and $\mathcal{O}(2,1)$, whose singular set in the Hopf link, admit infinitely many fibrations with base orbifold   a sphere  with at most two cone points (Case 5); for each integer $b>1$ the orbifold $\mathcal{O}(b,\pm 1)$  admits a fibration with base orbifold a  sphere with three cone points (Case 1); for each even integer $b>1$ the orbifold  $\mathcal{O}(4b,\pm(1+2b))$ admits a fibration with base orbifold a  disk with one cone point and one corner reflector (Case 4).

\end{example}

\bibliographystyle{alpha}
\bibliography{ms-bibliography}

\end{document}